\documentclass[11pt]{article}

\usepackage{latexsym}
\usepackage{amsmath}
\usepackage{amssymb}
\usepackage{mathrsfs}
\usepackage{amsthm}
\usepackage{makeidx}
\usepackage{tikz}
\usepackage{bussproofs}
\usepackage{enumerate}

\usetikzlibrary{arrows}
\usetikzlibrary{calc}

\usepackage{fancyhdr}
\usepackage[Lenny]{fncychap}
\usepackage{amsmath}
\usepackage{amsthm}
\usepackage{amssymb}
\usepackage{amsfonts}
\usepackage{latexsym}
\usepackage{mathrsfs}
\usepackage{makeidx}
\usepackage{tikz}
\usepackage{graphicx}
\usepackage{enumerate}
\usepackage{tocloft}
\usepackage[all,cmtip]{xy}
\usepackage[colorlinks=true,
            linkcolor=blue,
            urlcolor=red,
            citecolor=cyan]{hyperref}
\usepackage{bussproofs}
\usepackage{stmaryrd}	
\usepackage{leqno}
\usepackage[geometry]{ifsym}					
\usepackage{rotating}						
\usepackage{amscd}
\usepackage[position=top]{subfig}						
\usepackage{cmll, scalerel}

\newtheorem{theorem}{Theorem}[section]
 \newtheorem{lemma}[theorem]{Lemma}
 \newtheorem{prop}[theorem]{Proposition}

 \theoremstyle{definition}
 \newtheorem{definition}[theorem]{Definition}

 \newtheorem{remark}[theorem]{Remark}

 \theoremstyle{remark}

\renewcommand{\phi}{\varphi}

\newcommand{\nomi}{\mathbf{i}}
\newcommand{\nomj}{\mathbf{j}}
\newcommand{\nomk}{\mathbf{k}}

\newif\ifmargincoms

\definecolor{amethyst}{rgb}{0.6, 0.4, 0.8}
\definecolor{dodgerblue}{rgb}{0.12, 0.56, 1.0}

\newcommand{\At}{\mathit{At}}

\newcommand{\A}{\mathbf{A}}

\newcommand{\cut}[1]{}

\margincomstrue

\title{Algebraic semantics for hybrid logics}
\author{Willem Conradie and Claudette Robinson\\
{\small Dept. Pure and Applied Mathematics, University of Johannesburg}\\
{\small wconradie@uj.ac.za \quad claudette.robinson574@gmail.com}}
\date{}
\begin{document}

\maketitle

\begin{abstract}
We introduce hybrid algebras as algebraic semantics for hybrid languages with nominals and, possibly, the satisfaction operator. We establish a duality between hybrid algebras and the descriptive two-sorted general frames of Ten Cate. We show that all axiomatic extensions of the basic hybrid logics, with or without the satisfaction operator, are complete with respect to their classes of hybrid algebras. Moreover, we show that by adding the usual non-orthodox rules to these logics, they become complete with respect to their classes of permeated hybrid algebras, corresponding to strongly descriptive two-sorted general frames.
\end{abstract}


\section*{Introduction}
Hybrid logics (\cite{Gargov:Goranko:1993, BlckbrnManifesto, areces05:_hybrid_logic}) extend modal logic with a second sort of atomic formulas, known as \emph{nominals}, which are constrained to range over singleton subsets of Kripke frames and thus act as names for states in models. The expressive power of hybrid languages is further enhanced by the addition of various other connectives which  capitalize on this naming power of the nominals.

Historically, hybrid logics can be traced back to the work of Arthur Prior \cite{Prior57, Prior67}. They were given their present form by Gargov and Goranko \cite{Gargov:Goranko:1993}, and have seen rapid development since the the late 1990s.

In recent years modal logic has benefited greatly from the use of algebraic methods and the utilization of the duality between its relational and algebraic semantics (see e.g. \cite{Goldblatt1989173} and \cite{VenemaChapter}). Strangely enough, given that modal and hybrid logic are such close cousins, there has been very little work on algebraic semantics for the latter. The only work in this direction of which we are aware is that of Litak \cite{Litak06:relmics}, who provides algebraic semantics for a very expressive hybrid language, which contains the `down arrow' binder. The algebras introduced by Litak are akin the cylindric algebras used as semantics for first-order logic.

In this paper we consider hybrid languages without binders and, as a result, our algebras are much simpler. In fact, the algebraic semantics we introduce will be based on structures we call \emph{hybrid algebras}, which are adapted versions of the Boolean algebras with operators (BAOs) which form the usual algebraic semantics of modal logic.

We consider three hybrid languages, namely, the language obtained by adding nominals to the basic modal language and the two languages obtained by additionally adding, respectively, the satisfaction operator $@$ and the universal modality $\mathsf{A}$. Our main results are general completeness results for axiomatic extensions of the basic logics associated with these languages, with respect to the corresponding classes of hybrid algebras.

Stepping through the looking glass of duality, the relational duals of hybrid algebras are already to be found in the PhD thesis of Balder ten Cate \cite{TenCate:Phd:Thesis}, in the form of \emph{descriptive two sorted  general frames}. Our inspiration for the definitions of hybrid algebras comes from these two sorted general frames. In fact, the completeness results presented in this paper could alternatively be obtained by properly establishing the duality between hybrid algebras and descriptive two sorted  general frames, and then appealing to ten Cate's completeness results in \cite{TenCate:Phd:Thesis}. However, in our opinion, there is much value in presenting these results purely algebraically: firstly, it shows that a purely algebraic approach to the semantics of hybrid logic is feasible and fruitful; secondly, it provides the opportunity to develop and showcase a number of techniques and constructions on hybrid algebras that will prove very useful when this semantics is used to derive other results for hybrid logics like Sahlqvist-type theorems \cite{ConRob2015} and finite model properties.

The paper is organized as follows: section \ref{sec:prelim} collects the necessary preliminaries on the syntax, relational semantics and axiomatics of the logics under consideration. In Section \ref{sec:AlgbraicSem} we introduce hybrid algebras and permeated hybrid algebras as semantics for hybrid languages and prove some basic preliminary propositions.  In Section \ref{sec:algebraic:comp} we present and prove our main results, namely the completeness of all axiomatic extensions of our basic logics with respect to the corresponding classes of (permeated) hybrid algebras. We conclude in Section \ref{sec:Conc}.

\section{Preliminaries} \label{sec:prelim}

In this section we collect some essential preliminaries.

\paragraph{Syntax.} Fix countably infinite disjoint sets \textsf{PROP} and \textsf{NOM} of \emph{propositional variables} (denoted $p, q, r, \ldots$) and \emph{nominals} (denoted $\nomi, \nomj, \nomk, \ldots$), respectively. Then the syntax of the languages $\mathcal{H}$, $\mathcal{H}(@)$ and $\mathcal{H}(\mathsf{E})$ is defined as follows:
\begin{eqnarray*}
\phi & ::= & \bot \mid p \mid \nomj \mid \neg\phi \mid \phi \wedge \psi \mid \Diamond \phi \\
\phi & ::= & \bot \mid p \mid \nomj \mid \neg\phi \mid \phi \wedge \psi \mid \Diamond \phi \mid @_{\nomj}\phi \\
\phi & ::= & \bot \mid p \mid \nomj \mid \neg\phi \mid \phi \wedge \psi \mid \Diamond \phi \mid \mathsf{E}\phi
\end{eqnarray*}
Here $p \in \textsf{PROP}$ and $\nomj \in \textsf{NOM}$. The Boolean connectives $\top, \vee, \rightarrow$ and $\leftrightarrow$ are defined as usual, and as usual, $\Box \phi := \neg\Diamond\neg \phi$ and $\mathsf{A}\phi := \neg\mathsf{E}\neg\phi$.

\paragraph{Relational semantics.} Like modal languages, the languages $\mathcal{H}$, $\mathcal{H}(@)$ and $\mathcal{H}(\mathsf{E})$ can be interpreted in various related structures. We will be concerned with Kripke frames, two sorted general frames, models, and hybrid algebras (to be defined later). A (\emph{Kripke}) \emph{frame} is a pair $\frak{F} = (W, R)$ such that $W$ is a non-empty set (called the \emph{domain}) of objects called states, and $R \subseteq W^2$ is a binary \emph{accessibility relation} on $W$. A (\emph{Kripke}) \emph{model} based on a frame $\frak{F} = (W, R)$ is a pair $\frak{M}$ = $(\frak{F}, V)$, where $V$ is a \emph{valuation} $V: \textsf{PROP} \cup \textsf{NOM} \rightarrow \mathcal{P}(W)$ such that $V(\nomi)$ is a singleton subset of $W$.

The \emph{truth} of a formula $\phi$ at a state $w$ in a Kripke model $\frak{M} = (W, V)$, denoted $\frak{M}, w \Vdash \phi$, is defined as usual. In particular, $\frak{M}, w \Vdash \nomi$ iff $V(\nomi) = \{ w\}$; $\frak{M}, w \Vdash \Diamond\phi$ iff there exists $v$ such that $wRv$ and $\frak{M}, v \Vdash \phi$; $\frak{M}, w \Vdash @_{\nomj}\phi$ iff $\frak{M}, v \Vdash \phi$ and $V(\nomj) = \{v\}$; $\frak{M}, w \Vdash \mathsf{E}\phi$ iff there is some state $v \in W$ such that $\frak{M}, v \Vdash \phi$. Global truth in a model and local and global validity in frames are defined as usual.

Recall that a general frame for the basic modal language (see e.g.  \cite{BdRV01}) is a triple $\frak{g} = (W, R, A)$ such that $(W, R)$ is a Kripke frame, $A$ is a non-empty collection of subsets of $W$ (called \textit{admissible subsets}) which is closed under finite intersection, relative complement, and under the operation $\langle R\rangle$ defined by $\langle R\rangle X := \{w \in W \mid \exists v \in X \mbox{ such that } wRv \}$. An \emph{admissible valuation} on $\frak{g}$ is a map $V: \mathsf{PROP} \rightarrow A$. We can think of Kripke frames as the special case of general frames for which $A = \mathcal{P}(W)$.

In order to interpret hybrid languages on general frames, suitable provision must be made for the nominals. The \emph{two-sorted general frames} introduced by Ten Cate in \cite{TenCate:Phd:Thesis} do this in a natural way. Specifically, a \emph{two-sorted general frame} is a structure $\frak{g} = (W, R, A, B)$, where $(W, R, A)$ is a general frame, and $B \subseteq W$ is a non-empty set  such that $\{ w \} \in A$ for all $w \in B$, called the  \emph{admissible points of $\frak{g}$}. An \emph{admissible valuation} on $\frak{g}$ is then a map $V: \mathsf{PROP} \cup \mathsf{NOM} \rightarrow A$ such that $V(\nomi) \in \{ \{ w \} \mid w \in B \}$ for each nominal $\nomi \in \mathsf{NOM}$. A \textit{model based on $\frak{g} = (W, R, A, B)$} is a Kripke model $(W, R, V)$, where $V$ is an admissible valuation on $\frak{g}$. A hybrid formula $\phi$ is valid at a point $w$ in $\frak{g}$, written $\frak{g}, w \Vdash \phi$, if $(W,R,V), w \Vdash \phi$ for every admissible valuation $V$ on $\frak{g}$.

Recall (from e.g. \cite{BdRV01}) that a general frame $(W, R, A)$ is said to be \emph{differentiated}, if for all $w, v \in W$ with $w \neq v$, there exists an $a \in A$ such that $w \in a$ and $v \notin a$. A general frame $(W, R, A)$ is \emph{tight}, if for all $u, v \in W$, it is the case that $uRv$ iff $\forall a \in A(v \in a \Longleftrightarrow u \in \langle R\rangle a)$. It is \emph{compact}, if $\bigcap A_0 \neq \varnothing$ for every subset $A_0$ of $A$ which has the finite intersection property. We say that $(W, R, A)$ is \emph{descriptive} if it is differentiated, tight and compact.

Following \cite{TenCate:Phd:Thesis}, we say that a two-sorted general frame $(W, R, A, B)$ is \emph{descriptive} if the associated general frame $(W, R, A)$ is descriptive. Finally, $(W, R, A, B)$ is \emph{strongly descriptive} if it is descriptive, and satisfies the following two conditions: (i) for all $a \in A$, if $a \neq \varnothing$, then there is a $w \in B$ such that $w \in a$; (ii) for all $a \in A$ and $u \in B$, if $\{v \in a \mid uRv\} \neq \varnothing$, then there is a $w \in B$ such that $w \in a$ and $uRw$.

\paragraph{Logics.} We now recall axiomatizations of the minimal hybrid logics in the languages $\mathcal{H}$, $\mathcal{H}(@)$ and $\mathcal{H}(\mathsf{E})$. The systems given here are based on those in \cite{BdRV01} and \cite{TenCate:Phd:Thesis}. We will use the notation $\Diamond^{n}$ with $n \in \mathbb{N}$ to denote a string of $n$ consecutive $\Diamond$'s. The notation $\Box^{n}$ is defined similarly.

The \textit{minimal normal hybrid logic} $\textbf{H}$ is the smallest set of $\mathcal{H}$-formulas containing all propositional tautologies and the axioms in Table \ref{Table:Axiomatization:H}, and which is closed under the inference rules in Table \ref{Table:Axiomatization:H}, except for ($\mathit{Name}$) and ($\mathit{Paste}$). $\textbf{H}^{+}$ is defined similarly, closing in addition under ($\mathit{Name}$) and ($\mathit{Paste}$). If $\Sigma$ is a set of $\mathcal{H}$-formulas, then $\mathbf{H}\oplus\Sigma$ and $\mathbf{H}^{+}\oplus\Sigma$ are the normal hybrid logics \emph{generated} by $\Sigma$.

By \emph{sorted substitution} we mean uniform substitution of formulas for propositional variables and nominals for nominals.

\begin{table}[h!]
\begin{center}
\begin{tabular}{|ll|}
\hline
Axioms: & \\
& \\
($\mathit{Taut}$) & $\vdash \phi$ for all classical propositional tautologies $\phi$. \\
($\mathit{Dual}$) & $\vdash \Diamond p \leftrightarrow \neg\Box\neg p$ \\
($K$) & $\vdash \Box(p \rightarrow q) \rightarrow (\Box p \rightarrow \Box q)$ \\
($\mathit{Nom}$) & $\vdash \Diamond^n(\nomi \wedge p) \rightarrow \Box^m(\nomi \to p)$ for all $n, m \in \mathbb{N}$. \\
& \\
Rules of inference:  & \\
& \\
($\mathit{Modus}$ $\mathit{ponens}$) & If $\vdash \phi \to \psi$ and $\vdash \phi$, then
$\vdash \psi$. \\
($\mathit{Sorted}$ $\mathit{substitution}$) & $\vdash \phi'$ whenever $\vdash \phi$, where $\phi'$ is obtained from $\phi$ by sorted \\
& substitution. \\
($\mathit{Nec}$) & If $\vdash \phi$, then $\vdash \Box \phi$. \\
($\mathit{NameLite}$) & If $\vdash \neg\nomi$, then $\vdash \bot$. \\
($\mathit{Name}$) & If $\vdash \nomi \to \phi$, then $\vdash \phi$ for $\nomi$ not occurring in $\phi$. \\
($\mathit{Paste}$) & If $\vdash \Diamond^{n}(\nomi \wedge \Diamond(\nomj \wedge \phi)) \to \psi$, the $\vdash \Diamond^{n}(\nomi \wedge \Diamond\phi) \to \psi$ for $n \in \mathbb{N}$, \\
& $\nomi \neq \nomj$, and $\nomj$ not occurring in $\phi$ and $\psi$. \\
\hline
\end{tabular}
\end{center}
\caption{Axioms and inference rules of $\textbf{H}$ and $\textbf{H}^{+}$} \label{Table:Axiomatization:H}
\end{table}

The role of ($\mathit{NameLite}$) is to render logics that derive $\neg\nomj$ for some nominal $\nomj$, inconsistent, reflecting
the fact that $\neg\nomj$ is not valid on any frame. As is not hard to see, without ($\mathit{NameLite}$), the logic $\mathbf{H}\oplus\{\neg\nomj\}$ would be consistent.

The rules ($\mathit{Name}$) and ($\mathit{Paste}$), as well as their $@$ and $\mathsf{E}$ versions, are known as `non-orthodox'
rules because of their syntactic side-conditions. It is well known that these rules are admissible in the minimal hybrid logics obtained by omitting them. However, they are needed in order to obtain the well-known general completeness result regarding extensions
with pure axioms (see \cite{BdRV01} and \cite{TenCate:Phd:Thesis}).

The \emph{minimal normal hybrid logic} $\textbf{H}(@)$ is the smallest set of $\mathcal{H}(@)$-formulas containing all propositional tautologies and the axioms in Table \ref{Table:Axiomatization:H@}, and which is closed under the inference rules in Table
\ref{Table:Axiomatization:H@}, except for ($\mathit{Name}_@$) and ($\mathit{BG}_@$). $\textbf{H}^{+}(@)$ is defined similarly, closing in addition under ($\mathit{Name}_@$) and ($\mathit{BG}_@$). If $\Sigma$ is a set of $\mathcal{H}(@)$-formulas, then $\mathbf{H}\oplus\Sigma$ and $\mathbf{H}^{+}\oplus\Sigma$ are the normal hybrid logics \emph{generated} by $\Sigma$.

\begin{table}[h!]
\begin{center}
\begin{tabular}{|ll|}
\hline

Axioms: & \\
& \\
($\mathit{Taut}$) & $\vdash \phi$ for all classical propositional tautologies $\phi$. \\
($K$) & $\vdash \Box(p \rightarrow q) \rightarrow (\Box p \rightarrow \Box q)$ \\
($\mathit{Dual}$) &  $\vdash \Diamond p \leftrightarrow \neg \Box \neg p$ \\
($K_{@}$) & $\vdash @_{\nomj}(p \to q) \rightarrow (@_{\nomj}p \to @_{\nomj}q)$ for all $\nomj \in \mathsf{NOM}$. \\
($\mathit{Selfdual}$) & $\vdash \neg @_{\nomj}p \leftrightarrow @_{\nomj}\neg p$ for all $\nomj \in \mathsf{NOM}$. \\
($\mathit{Intro}$) & $\vdash \nomj \wedge p \to @_{\nomj}p$ for all $\nomj \in \mathsf{NOM}$. \\
($\mathit{Ref}$) & $\vdash @_{\nomj}\nomj$ for all $\nomj \in \mathsf{NOM}$. \\
($\mathit{Agree}$) & $\vdash @_{\nomi}@_{\nomj}p \to @_{\nomj}p$ for all $\nomi, \nomj \in \mathsf{NOM}$. \\
($\mathit{Back}$) & $\vdash \Diamond @_{\nomj}p \to @_{\nomj}p$ for all $\nomj \in \mathsf{NOM}$. \\
& \\
Rules of inference:  & \\
& \\
($\mathit{Modus}$ $\mathit{ponens}$) & If $\vdash \phi \to \psi$ and $\vdash \phi$, then
$\vdash \psi$. \\
($\mathit{Sorted}$ $\mathit{substitution}$) & $\vdash \phi'$ whenever $\vdash \phi$, where $\phi'$ is obtained from $\phi$ by sorted \\
& substitution. \\
($\mathit{Nec}$) & If $\vdash \phi$, then $\vdash \Box \phi$. \\
($\mathit{Nec}_@$) & If $\vdash \phi$, then $\vdash @_{\nomj}\phi$. \\
($\mathit{Name}_@$) & If $\vdash @_{\nomj}\phi$, then $\vdash \phi$ for $\nomj$ not occurring in $\phi$. \\
($\mathit{BG}_@$) & If $\vdash @_{\nomi}\Diamond\nomj \wedge @_{\nomj}\phi \to \psi$, then $\vdash @_{\nomi}\Diamond\phi\to \psi$ \\
& for $\nomj \neq \nomi$ and $\nomj$ not occurring in $\phi$ and $\psi$. \\
\hline
\end{tabular}
\end{center}
\caption{Axioms and inference rules of $\textbf{H}(@)$ and $\textbf{H}^{+}(@)$} \label{Table:Axiomatization:H@}
\end{table}

The \textit{minimal normal hybrid logic} $\textbf{H}(\mathsf{E})$ is the smallest set of $\mathcal{H}(\mathsf{E})$-formulas containing all propositional tautologies and the axioms in Table \ref{Table:Axiomatization:HE}, and which is closed under the inference rules in Table \ref{Table:Axiomatization:HE}, except for ($\mathit{Name}_{\mathsf{E}}$), ($\mathit{BG}_{\mathsf{E}\Diamond}$) and ($\mathit{BG}_{\mathsf{E}\mathsf{E}}$). $\textbf{H}^{+}(\mathsf{E})$ is defined in the same way, closing in addition under ($\mathit{Name}_{\mathsf{E}}$), ($\mathit{BG}_{\mathsf{E}\Diamond}$) and ($\mathit{BG}_{\mathsf{E}\mathsf{E}}$). If $\Sigma$ is a set of $\mathcal{H}(\mathsf{E})$-formulas, then $\mathbf{H}(\mathsf{E})\oplus\Sigma$ and $\mathbf{H}^{+}(\mathsf{E})\oplus\Sigma$ are the normal hybrid logics generated by $\Sigma$.

\begin{table}[h!]
\begin{center}
\begin{tabular}{|ll|}
\hline

Axioms: & \\
& \\
($\mathit{Taut}$) & $\vdash \phi$ for all classical propositional tautologies. \\
($K$) & $\vdash \Box(p \rightarrow q) \rightarrow (\Box p \rightarrow \Box q)$ \\
($\mathit{Dual}$) &  $\vdash \Diamond p \leftrightarrow \neg \Box \neg p$ \\
($K_{\mathsf{A}}$) & $\vdash \mathsf{A}(p \to q) \rightarrow (\mathsf{A}p \to \mathsf{A}q)$ \\
($\mathit{Dual}_{\mathsf{A}}$) & $\vdash \mathsf{E} p \leftrightarrow \neg \mathsf{A} \neg p$ \\
($\mathit{Incl}_{\nomj}$) & $\vdash \mathsf{E}\nomj$ \\
($\mathit{Nom}_{\mathsf{E}}$) & $\vdash \textsf{E}(\nomi \wedge p) \to \textsf{A}(\nomi \to p)$ \\
($T\mathsf{E}$) & $\vdash p \to \textsf{E}p$ \\
($4\mathsf{E}$) & $\vdash \textsf{E}\textsf{E}p \to \textsf{E}p$ \\
($B\mathsf{E}$) & $\vdash p \to \textsf{AE}p$ \\
($\mathit{Incl}_{\Diamond}$) & $\vdash \Diamond p \to \textsf{E}p$ \\
& \\
Rules of inference:  & \\
& \\
($\mathit{Modus}$ $\mathit{ponens}$) & If $\vdash \phi \to \psi$ and $\vdash \phi$, then
$\vdash \psi$. \\
($\mathit{Sorted}$ $\mathit{substitution}$) & $\vdash \phi'$ whenever $\vdash \phi$, where $\phi'$ is obtained from $\phi$ by sorted \\
& substitution. \\
($\mathit{Nec}$) & If $\vdash \phi$, then $\vdash \Box \phi$. \\
($\mathit{Nec}_\mathsf{A}$) & If $\vdash \phi$, then $\vdash \mathsf{A}\phi$. \\
($\mathit{Name}_{\mathsf{E}}$) & If $\vdash \nomi \to \phi$, then $\vdash \phi$ for $\nomi$ not occurring in $\phi$. \\
($\mathit{BG}_{\mathsf{E}\Diamond}$) & If $\vdash \mathsf{E}(\nomi \wedge \Diamond\nomj) \wedge \mathsf{E}(\nomj \wedge \phi) \to \psi$, then $\vdash \mathsf{E}(\nomi \wedge \Diamond\phi) \to \psi$ for \\
& $\nomi \neq \nomj$ and $\nomj$ not occurring in $\phi$ and $\psi$. \\
($\mathit{BG}_{\mathsf{E}\mathsf{E}}$) & If $\vdash \mathsf{E}(\nomi \wedge \mathsf{E}\nomj) \wedge \mathsf{E}(\nomj \wedge \phi) \to \psi$, then $\vdash \mathsf{E}(\nomi \wedge \mathsf{E}\phi) \to \psi$ for \\
& $\nomi \neq \nomj$ and $\nomj$ not occurring in $\phi$ and $\psi$. \\
\hline
\end{tabular}
\end{center}
\caption{Axioms and inference rules of $\mathbf{H}(\mathsf{E})$ and $\textbf{H}^{+}(\mathsf{E})$} \label{Table:Axiomatization:HE}
\end{table}

\begin{remark}
One naturally wonders how the logics with and without the additional inference rules compare. The minimal hybrid logics $\mathbf{H}$ and $\mathbf{H}^{+}$ have the same theorems, since both are sound and strongly complete with respect to the class of all frames, see e.g. \cite{BdRV01}. This picture changes when we extend these logics with additional axioms. Consider, for example, the set $\Sigma = \{\nomj \to \Box\bot\}$ and the formula $\phi = \Diamond\top$. Let $\frak{g} = (W, R, A, B)$ be a descriptive two-sorted  general frame with $W = \{u, v\}$, $R = \{(u,u)\}$, $A = \mathcal{P}(W)$ and $B = \{v\}$ (see Figure \ref{Comparison:H:H+}). Then $\frak{g}$ validates the members of $\Sigma$, and furthermore, $\Diamond\top$ is satisfied at $u$. By the soundness of $\mathbf{H}\oplus\Sigma$ with respect to its class of descriptive two-sorted  general frames (see above), we have $\mathbf{H}\oplus\Sigma \not\vdash \Box\bot$. On the other hand, by applying the ($\mathit{Name}$) rule to $\nomj \to \Box\bot$, we see that $\mathbf{H}^+\oplus\Sigma \vdash \Box\bot$.

Analogous arguments, using the same two-sorted descriptive general frame $\frak{g}$, show that $\mathbf{H}(@)\oplus \{@_{\nomj }\Box \bot \} \not \vdash \Box \bot$ while $\mathbf{H}^{+}(@)\oplus \{@_{\nomj }\Box \bot \} \vdash \Box \bot$ and that $\mathbf{H}(\mathsf{E})\oplus \{\nomj \to \Box\bot \} \not \vdash \Box \bot$ while $\mathbf{H}^{+}(\mathsf{E})\oplus \{\nomj \to \Box\bot\} \vdash \Box \bot$. 

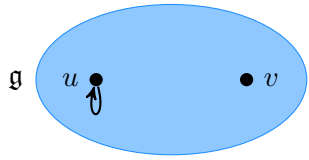
\begin{figure}[h!]
\centering
\begin{tikzpicture}[[->,>=stealth', node/.style={circle, draw, fill=black}, scale=1]

\draw[draw=dodgerblue, fill=dodgerblue!50] (0,0) ellipse (1.8 and 1);

\foreach \nom/\x/\y/\pos in {u/-1/0/left, v/1/0/right}
{
\node[circle,fill=black!100,minimum size=5pt,inner sep=0pt, label=\pos:$\nom$] (N-\nom)  at (\x, \y) {};
}

\draw[line width=0.8pt, ->, loop below=50] (N-u) to node {} (N-u);

\node (F) at (-1.7, 0) [label=left:$\frak{g}$] {};

\end{tikzpicture}
\caption{A two-sorted general frame satisfying the formula $\Diamond\top$.} \label{Comparison:H:H+}
\centering
\end{figure}
\end{remark}

\paragraph{General completeness with respect to two sorted general frames.} Ten Cate \cite{TenCate:Phd:Thesis} has shown that, for every set $\Sigma$ of $\mathcal{H}$-formulas (respectively, $\mathcal{H}(@)$-formulas, $\mathcal{H}(\mathsf{E})$-formulas), the logic $\mathbf{H}\oplus\Sigma$ (respectively, $\mathbf{H}(@)\oplus\Sigma$, $\mathbf{H}(\mathsf{E})\oplus\Sigma$) is strongly complete with respect to the class of descriptive two-sorted general frames validating its axioms. Moreover,  $\mathbf{H}^{+}\oplus\Sigma$ (respectively, $\mathbf{H}^{+}(@)\oplus\Sigma$, $\mathbf{H}^{+}(\mathsf{E})\oplus\Sigma$) is strongly complete with respect to the class of \emph{strongly} descriptive two-sorted general frames validating its axioms.

\paragraph{Boolean algebras with operators.}

Boolean algebras with additional operators (BAOs) offer a nuatural and standard interpretation of  modal languages.  A \emph{Boolean algebra with operator} is an algebra $\mathbf{A} = (A, \wedge, \vee, \neg, \bot, \top, f)$
such that $(A, \wedge, \vee, \neg, \bot, \top)$ is a Boolean algebra and $f$ is an \emph{operator}, i.e., a function from $A$ to $A$ satisfying the following:
\begin{description}

\item[$(\mathit{normality})$] \index{normality} $f(\bot) = \bot$, and

\item[$(\mathit{additivity})$] \index{additivity} $f(a \vee b) = f(a) \vee f(b)$.

\end{description}

An operator $f$ on a Boolean algebra is \emph{monotonic}, if $a \leq b$ implies $f(a) \leq f(b)$. All operators are monotonic. To see this, assume $a \leq b$. Then $a \vee b = b$, so $f(a) \vee f(b) = f(a \vee b) = f(b)$, which means $f(a) \leq f(b)$. We say that operators have the property of \emph{monotonicity}.

Let $\mathbf{A} = (A, \wedge, \vee, \neg, \bot, \top, f)$ be a BAO. An \emph{assignment} on $\mathbf{A}$ is a function $v$: $\mathsf{PROP} \to A$ associating an element of $A$ with each propositional variable in $\mathsf{PROP}$. Given such an assignment $v$, we calculate the \emph{meaning} $\tilde{v}(t)$ of a term $t$ as follows: $\tilde{v}(\bot) = \bot, \tilde{v}(p) = v(p), \tilde{v}(\neg \psi) = \neg\tilde{v}(\psi), \tilde{v}(\psi_1 \wedge \psi_2) = \tilde{v}(\psi_1) \wedge \tilde{v}(\psi_2)$ and $\tilde{v}(\Diamond\psi) = f(\tilde{v}(\psi))$.

We say that an equation $\phi \approx \psi$ is \emph{true} in a BAO $\mathbf{A}$ (denoted $\mathbf{A} \models \phi \approx \psi$), if for all assignments $\theta$, $\tilde{\theta}(\phi) = \tilde{\theta}(\psi)$.

Every normal modal logic is sound and complete with respect to the class of all BAOs which validate its axioms (see e.g. \cite{BdRV01}). This is no surprise as every normal logic is also sound and strongly complete with respect to its class of descriptive general frames \cite{Gold76}, and BAOs and descriptive general frames are duals of each other (see e.g. \cite{BdRV01}).

\paragraph{Canonical extensions.} Recall that the \emph{canonical extension} of a BAO $\A$ (see e.g., \cite{VenemaChapter}) is complete and atomic and is the unique superalgebra $\A^{\delta}$ of $\A$ satisfying:

\begin{description}

\item[(density)] every element of $\A^{\delta}$ can be expressed as both a join of meets and as a meet of joins of elements of $\A$, and

\item[(compactness)] for all subsets $S \subseteq A$ and $T \subseteq A$, if $\bigwedge S \leq \bigvee T$ in $\A^{\delta}$, then there exist finite sets $S_0 \subseteq S$ and $T_0 \subseteq T$ such that $\bigwedge S_0 \leq \bigvee T_0$.
\end{description}

\noindent
Arbitrary meets (joins) of elements of $\A$ form the \emph{closed} (\emph{open}) \emph{elements} of $\A^{\delta}$. It is not difficult to see that $\At \A \subseteq \At \A^{\delta}$.

\paragraph{Adjoint pairs.} In our general completeness proofs later we will give algebraic constructions that make crucial use of the adjoint of $\Box$, so let us recall some relevant preliminaries on adjoints. In what follows, $\mathbf{A}$ and $\mathbf{B}$ are two complete BAOs. The monotone maps $f$: $\mathbf{A} \to \mathbf{B}$ and $g$: $\mathbf{B} \to \mathbf{A}$ form an \emph{adjoint pair} (denoted $f \dashv g$), if for all $a$ in $\mathbf{A}$ and $b$ in $\mathbf{B}$, $f(a) \leq b$ iff $a \leq g(b)$. If $f \dashv g$, $f$ is called the \emph{left adjoint} of $g$, while $g$ is called the \emph{right adjoint} of $f$.

An important property of adjoint pairs is that if a map is completely join-preserving (meet-preserving), then we can compute its right (left) adjoint pointwise from the map itself and the order relation on the BAO: for monotone maps $f$: $A \to B$ and $g$: $B \to A$ such that $f \dashv g$, $f(a) = \bigwedge\{b \in B \mid a \leq g(b)\}$ and $g(b) = \bigvee\{a \in A \mid f(a) \leq b\}$. Moreover, for any map $f$: $A \to B$, $f$ is completely join-preserving (completely meet-preserving) iff it has a right adjoint (left adjoint). For the proofs of this see \cite{DaPriest02}.

We know that $\mathbf{A}^{\delta}$ is perfect (see e.g. \cite{VenemaChapter}), so the operations $\Diamond$ and $\Box$ are completely join- and meet-preserving, respectively, and therefore have right and left adjoints, respectively. We will denote the right adjoint of $\Diamond$ by $\Box^{-1}$, and the left adjoint of $\Box$ by $\Diamond^{-1}$. Alternatively, $\mathbf{A}^{\delta}$ is isomorphic to the complex algebra of the ultrafilter frame of $\mathbf{A}$, and $\Box^{-1}$ and $\Diamond^{-1}$ are the `tense' modalities interpreted with the inverse of the accessibility relation in that frame.

\section{Algebraic semantics} \label{sec:AlgbraicSem}

In this section we introduce the \emph{hybrid algebras} which will form the algebraic semantics for the languages $\mathcal{H}$, $\mathcal{H}(@)$. We also introduce the so-called \emph{grounded} and \emph{degenerate} versions of these structures and study certain truth preserving operations on them. We define \emph{orthodox interpretations} which treat nominals as constants, thus allowing us to fall back on much of the existing theory for modal logic as interpreted in BAOs.

\subsection{Hybrid algebras for the language $\mathcal{H}$}

\cut{It is a well-known fact that most of the familiar logical systems have a (natural) algebraic semantics. In this section, we develop algebraic semantics for the hybrid languages discussed in Section \ref{sec:prelim}. But in order to do that, we must first know which kinds of algebras are relevant. A first step towards an answer comes from Ten Cate's work in \cite{TenCate:Phd:Thesis}. He studied completeness theory for the axiomatizations of these languages. He proved, among other things, a general completeness result for the logics without the additional rules with respect to descriptive two sorted general frames, as well as completeness for the logics with the additional rules with respect to strongly descriptive two sorted general frames. This suggests using algebras corresponding to these two-sorted general frames. However, in the proof of his completeness result with respect to descriptive two-sorted general frames, he treats nominals as unary modalities. This suggests that there are two possible algebraic semantics for each of the hybrid languages discussed in Section \ref{sec:prelim}. One involves interpreting nominals as constants and then falling back on known results from algebraic logic. The second corresponds to two-sorted general frames where the nominals are seen as special variables ranging over a subset of the atoms of the algebra.
}

The first type of algebraic semantics for $\mathcal{H}$ is called an \textit{orthodox interpretation}. An \emph{orthodox interpretation} of $\mathcal{H}$ is a structure $\mathbf{A} = (A, \wedge, \vee, \neg, \bot, \top, \Diamond, \{s_\nomi\}_{\nomi \in \mathsf{NOM}})$, where $(A, \wedge, \vee, \neg, \bot, \top, \Diamond)$ is a BAO, and each $s_\nomi$ is the interpretation of the nominal $\nomi$ as a \emph{constant}. Moreover, $\mathbf{A}$ is required to validate the inequality $\Diamond^{n}(s_{\nomi} \wedge a) \leq \Box^{m}(\neg s_\nomi \vee a)$ for all $\nomi \in \mathsf{NOM}$ and $n, m \in \mathbb{N}$. We use the term ``orthodox'' since it is really the `standard' algebraic semantics for modal logics with constants. However, for us it is `non-standard' since it is not appropriately dual to the intended relational semantics of hybrid logic. Moreover, the rule ($\mathit{Sorted\; substitution}$) is not generally sound on orthodox interpretations, in the sense that for an orthodox interpretation $\mathbf{A}$ and $\mathcal{H}$-formula $\psi$, it may happen that $\mathbf{A} \models \psi \approx \top$ but that $\mathbf{A} \not \models \psi' \approx \top$ for some sorted substitution instance $\psi'$ of $\psi$. For example, consider the orthodox interpretation $\mathbf{A} = (\mathbf{2}, \Diamond, \{s_{\nomj}\}_{\nomj \in \textsf{NOM}})$, where $\mathbf{2}$ is the two element Boolean algebra, $\Diamond 0 = 0$, $\Diamond1 = 1$, $s_\nomj = 0$, and $s_\nomi = 1$ for $\nomi \neq \nomj$. Then $\mathbf{A} \models \Diamond\nomi \approx \top$ but $\mathbf{A} \not\models \Diamond\nomj \approx \top$. However, this will not be a concern to us, as in the ensuing we will always require that an orthodox interpretation validates (all theorems of a) logic $\mathbf{H}\oplus\Sigma$, which is by definition closed under sorted substitution already.

The second and main type of algebraic structures for $\mathcal{H}$ which we consider are \emph{hybrid algebras}, defined below. Unlike orthodox interpretations they do not hold nominals constant and they also enjoy a duality with two-sorted general frames.

\begin{definition}
A \emph{hybrid algebra} is a pair $\frak{A} = (\mathbf{A}, X_A)$, where $\mathbf{A} = (A, \wedge, \vee, \neg, \bot, \top, \Diamond)$ such that $(A, \wedge, \vee, \neg, \bot, \top, \Diamond)$ is a BAO containing at least one atom and $X_A$ is a non-empty subset of the set $At \mathbf{A}$ of atoms of $\mathbf{A}$. We will often refer to $X_A$ as a \textit{set of designated atoms} of the algebra. We also make the following convention: $\Box a := \neg\Diamond\neg a$. Finally, the the class of hybrid algebras will be denoted by $\mathsf{HA}$.

We say that a hybrid algebra $\frak{A} = (\A, X_A)$ is \emph{permeated}, if for each $\bot \neq a \in A$, there is an atom $x \in X_A$ such that $x \leq a$, and for all $x \in X_A$ and $a \in A$, if $x \leq \Diamond a$, then there exists a $y \in X_A$ such that $y \leq a$ and $x \leq \Diamond y$. We will denote the class of permeated hybrid algebras by $\mathsf{PHA}$.
\end{definition}


An \emph{assignment} on a hybrid algebra $\frak{A} = (\A, X_A)$ is a map $v$: $\mathsf{PROP} \cup \mathsf{NOM} \to A$ associating an element of $A$ with each propositional variable in $\mathsf{PROP}$ and an atom of $X_A$ with each nominal in $\mathsf{NOM}$. Given such an assignment $v$, we calculate the \emph{meaning} $\tilde{v}(t)$ of a term $t$ as follows: $\tilde{v}(\bot) = \bot$, $\tilde{v}(p) = v(p)$, $\tilde{v}(\nomj) = v(\nomj)$, $\tilde{v}(\neg \psi) = \neg\tilde{v}(\psi)$, $\tilde{v}(\psi_1 \wedge \psi_2) = \tilde{v}(\psi_1) \wedge \tilde{v}(\psi_2)$, and $\tilde{v}(\Diamond \psi) = \Diamond\tilde{v}(\psi)$.

An equation $\phi \approx \psi$ is \textit{true} in a hybrid algebra $\frak{A}$ (denoted $\frak{A} \models \phi \approx \psi$), if for all assignments $v$, $\tilde{v}(\phi) = \tilde{v}(\psi)$.  A set $E$ of equations is true in a hybrid algebra $\frak{A}$ (denoted $\frak{A} \models E$), if each equation in $E$ is true in $\frak{A}$.

We now turn our attention to products of hybrid algebras.
Let $\frak{A} = (\mathbf{A}, X_A)$ and $\frak{B} = (\mathbf{B}, X_B)$ be two hybrid algebras. The product $\frak{A} \times \frak{B}$ of $\frak{A}$ and $\frak{B}$ is given by $(\mathbf{A} \times \mathbf{B}, X_{A \times B})$, where $\mathbf{A} \times \mathbf{B}$ is defined in the usual way and $X_{A \times B} = \{(x, \bot^\frak{B}) \mid x \in X_A \} \cup \{(\bot^{\frak{A}}, y) \mid y \in X_B\}$.

Validity is \emph{not} generally preserved under products of hybrid algebras. This mirrors the fact that validity of hybrid formulas is not preserved under taking disjoint unions of frames. Consider the hybrid algebra $\frak{A} = (\mathbf{2}, \Diamond,\{1\})$, where $\mathbf{2}$ is the two element Boolean algebra, $\Diamond 0 = 0$, and $\Diamond1 = 1$. Then $\frak{A} \models \Diamond \nomi \approx \top$ but $\frak{A}^2 \not \models \Diamond \nomi \approx \top$.

We can easily fix this by simply adding $\bot$ to the sets of designated atoms. A \textit{grounded hybrid algebra} is just like a hybrid algebra $\frak{A} = (\mathbf{A}, X_A)$, except that the bottom element of the algebra is also included in the set $X_A$. $\mathcal{H}$-equations are interpreted in grounded hybrid algebras as they are in hybrid algebras with nominals ranging over the elements in the designated set of atoms and $\bot$. Given a hybrid algebra $\frak{B} = (\mathbf{B}, X_B)$, the \emph{associated grounded hybrid algebra} is the structure $\frak{B}_0 = (\mathbf{B}, X_B \cup \{\bot\})$.

Although validity is not preserved under taking products of hybrid algebras, we can show that if each of the associated grounded hybrid algebras of two hybrid algebras validates a $\mathcal{H}$-formula, then the product of the original hybrid algebras does too. But first, recall that the \textit{projection map} on the $i$th coordinate of $A_1 \times A_2$ is the map $\pi_i$: $A_1 \times A_2 \to A_i$ defined by $\pi_i(a_1, a_2) = a_i$.

\begin{prop}\label{Prop:Product:Pres}
If $\frak{A}$ and  $\frak{B}$ are hybrid algebras such that $\frak{A}_0 \models \phi \approx \psi$ and $\mathfrak{B}_0 \models \phi \approx \psi$, then $\frak{A} \times \frak{B} \models \phi \approx \psi$.
\end{prop}

\begin{proof}
We prove the contrapositive, so assume $\frak{A} \times \frak{B} \not\models \phi \approx \psi$. Then there is an assignment \break $\nu$: $\mathsf{PROP} \cup \mathsf{NOM} \to A \times B$ such that $\nu(\phi) \neq \nu(\psi)$. But then $\pi_1(\nu(\phi)) \neq \pi_1(\nu(\psi))$ or $\pi_2(\nu(\phi)) \neq \pi_2(\nu(\psi))$. If $\pi_1(\nu(\phi)) \neq \pi_1(\nu(\psi))$, consider the assignment $\iota$: $\mathsf{PROP} \cup \mathsf{NOM} \to A$ defined by $\iota(p) = \pi_1(\nu(p))$ and $\iota(\nomj) = \pi_1(\nu(\nomj))$ for each $p \in \mathsf{PROP}$ and $\nomj \in \mathsf{NOM}$. Note that $\iota(\nomj)$ might be $\bot$ since some of the atoms in $X_{A \times B}$ are of the form $(\bot, y)$. Using structural induction, we can show that $\iota(\gamma) = \pi_1(\nu(\gamma))$ for all $\mathcal{H}$-formulas $\gamma$. Hence, $\iota(\phi) = \pi_1(\nu(\pi)) \neq \pi_1(\nu(\psi)) = \iota(\psi)$, and so $\frak{A}_0 \not\models \phi \approx \psi$. The case where $\pi_2(\nu(\phi)) \neq \pi_2(\nu(\psi))$ is similar.
\end{proof}





In the next section, we would need to take the product of two hybrid algebras $\frak{A} = (\mathbf{A}, X_A)$ and $\frak{B} = (\mathbf{B}, X_B)$, where either $X_A$ or $X_B$ is empty. But since this goes beyond the definition of a hybrid algebra, we will refer to these structures as \emph{degenerate hybrid algebras}. A \textit{degenerate hybrid algebra} is a pair $\frak{A} = (\mathbf{A}, X_A)$, where $\mathbf{A}$ is a BAO and $X_A = \varnothing$. Given a degenerate hybrid algebra $\mathfrak{A} = (\mathbf{A}, X_A)$, the \emph{associated grounded degenerate hybrid algebra} is the structure $\mathfrak{A}_0 = (\mathbf{A}, \{\bot\})$. We then have the following useful preservation result:

\begin{prop}\label{Product:degenerate:Pres:Lemma}
Let $\frak{A}$ be a degenerate hybrid algebra and $\frak{B}$ a hybrid algebra. If $\frak{A}_0 \models \phi \approx \psi$ and $\mathfrak{B} \models \phi \approx \psi$, then $\frak{A} \times \frak{B} \models \phi \approx \psi$.
\end{prop}

\begin{proof}
The proof is similar to that of Proposition \ref{Prop:Product:Pres}
\end{proof}

Later we will need the fact that the product of two permeated hybrid algebras is also permeated. We now show that this is indeed the case.

\begin{prop} \label{Prop:Perm:preserved:prod}
Let $\frak{A} = (\mathbf{A}, X_A)$ and $\frak{B} = (\mathbf{B}, X_B)$ be two permeated hybrid algebras. Then $\frak{A} \times \frak{B}$ is also permeated.
\end{prop}
\begin{proof}
For the first condition, let $(a, b) \in A \times B$ such that $(a, b) \neq (\bot, \bot)$. Then $a \neq \bot$ or $b \neq \bot$. Assume $a \neq \bot$. Now, since $\frak{A}$ is permeated, there is some $x \in X_A$ such that $x \leq a$. So we know that $(x, \bot) \leq (a, b)$. Furthermore, $x \in X_A$, hence $(x, \bot) \in X_{A \times B}$ by definition. Similarly for $b \neq \bot$. For the second condition, we have two consider two cases cases:\

{\bf Case 1:} $(\bot, y) \in X_{A \times B}$ such that $(\bot, y) \leq \Diamond(a, b)$. Then $(\bot, y) \leq (\Diamond a, \Diamond b)$, and so, $y \leq \Diamond b$. But we know that $\frak{B}$ is permeated, so there is a $y' \in X_B$ such that $y' \leq b$ and $y \leq \Diamond y'$. Hence, $(\bot, y') \leq (a, b)$ and $(\bot, y) \leq (\bot, \Diamond y') = (\Diamond\bot, \Diamond y') = \Diamond(\bot, y')$. Furthermore, since $y' \in X_B$, $(\bot, y') \in X_{A \times B}$.

{\bf Case 2:} $(x, \bot) \in X_{A \times B}$ such that $(x, \bot) \leq \Diamond(a, b)$. Similar to Case 1.
\end{proof}

\subsection{Hybrid algebras for the language $\mathcal{H(@)}$}

As for the language $\mathcal{H}$, we present two possible algebraic semantics for $\mathcal{H}(@)$. Again the first involves interpreting  nominals as constants, and as for $\mathcal{H}$, we will refer to these algebras as orthodox interpretations. An \emph{orthodox interpretation} of $\mathcal{H}(@)$ is an algebra $\mathbf{A} = (A, \wedge, \vee, \neg, \bot, \top, \Diamond, @, \{s_\nomi\}_{\nomi \in \mathsf{NOM}})$, where $(A, \wedge, \vee, \neg, \bot, \top, \Diamond)$ is a BAO, $@$ is a binary operator, each $s_\nomi$ is the interpretation of the nominal $\nomi$ as a \emph{constant}, and $\mathbf{A}$ is required to validate the following for all $\nomi, \nomj \in \mathsf{NOM}$:
\begin{align*}
@_{s_{\nomi}}(\neg a \vee b) &\leq \neg @_{s_{\nomi}} a \vee @_{s_{\nomi}} b & \neg @_{s_{\nomi}}a &= @_{s_{\nomi}}\neg a & @_{s_{\nomi}}@_{s_{\nomj}}a &\leq @_{s_{\nomj}}a \\
@_{s_{\nomi}}s_{\nomi} &= \top & s_{\nomi} \wedge a &\leq @_{s_{\nomi}}a & \Diamond @_{s_{\nomi}}a &\leq @_{s_{\nomi}}a
\end{align*}

The behaviour of the $@$ operator in an orthodox interpretation of $\mathcal{H}(@)$ is characterized by Proposition \ref{Lemma:@:BehavesCorrectly:ortho}. To prove this proposition, we make use of the fact that, in orthodox interpretations, the satisfaction operator preserves finite (including empty) meets and joins in its second coordinate, and is consequently also monotone in that coordinate. The proofs of both these facts are left for the reader.

\begin{prop} \label{Lemma:@:BehavesCorrectly:ortho}
Let $\mathbf{A}$ be an orthodox interpretation of $\mathcal{H}(@)$, and let $a$ be an element of $A$ and $s_{\nomi}$ the constant interpretation of $\nomi$. Then $@_{s_{\nomi}}a = \top$ iff $s_{\nomi} \leq a$ and $@_{s_{\nomi}}a = \bot$ iff $s_{\nomi} \leq \neg a$.
\end{prop}

\begin{proof}
First, assume $@_{s_{\nomi}}a = \top$. Then $\neg @_{s_{\nomi}}a = \bot$, and so $@_{s_{\nomi}}\neg a = \bot$. But $s_{\nomi} \wedge \neg a \leq @_{s_{\nomi}}\neg a$, so $s_{\nomi} \wedge \neg a \leq \bot$, which implies that $s_{\nomi} \leq a$.

For the converse, assume $s_{\nomi} \leq a$. Then $s_{\nomi} \vee a = a$, and so $@_{s_{\nomi}}(s_{\nomi} \vee a) = @_{s_{\nomi}}a$. Hence, since @ preserves finite joins in its second coordinate, $@_{s_{\nomi}}s_{\nomi} \vee @_{s_{\nomi}}a = @_{s_{\nomi}}a$. But $@_{s_{\nomi}}s_{\nomi} = \top$, so $@_{s_{\nomi}}a = \top$.

Next, assume $@_{s_{\nomi}}a = \bot$. Then we have $s_{\nomi} \wedge a \leq @_{s_{\nomi}} a = \bot$, so $s_{\nomi} \leq \neg a$, as required.

Conversely, suppose $s_{\nomi} \leq \neg a$. By the monotonicity of $@$, $@_{s_{\nomi}}s_{\nomi} \leq @_{s_{\nomi}}\neg a$, and so, since $@_{s_{\nomi}}s_{\nomi} = \top$, $@_{s_{\nomi}}\neg a = \top$. Hence, $\neg @_{s_{\nomi}}a = \top$, and so $@_{s_{\nomi}}a = \bot$.
\end{proof}

We now extend the hybrid algebras introduced in the previous subsection to accommodate the interpretation of the $@$ operator.


\begin{definition}
A hybrid $@$-algebra is a pair $\frak{A} = (\mathbf{A}, X_A)$, where $\mathbf{A} = (A, \wedge, \vee, \neg, \bot, \top, \Diamond, @)$ with $(A, \wedge, \vee, \neg, \bot, \top, \Diamond)$ a BAO containing at least one atom, $X_A$ is non-empty subset of atoms of $\mathbf{A}$, $@$ is a binary operator whose first coordinate ranges over $X_A$ and the second coordinate over all elements of the algebra, and for all $a, b \in A$ and all $x, y \in X_A$ the following hold:

\begin{center}
\begin{minipage}{7cm}

\begin{description}

\item[($\mathit{K@}$)] $@_x(\neg a \vee b) \leq \neg @_x a \vee @_x b$,

\item[($\mathit{self}$-$\mathit{dual}$)] $\neg @_{x}a = @_{x}\neg a$,

\item[($\mathit{agree}$)] $@_{x}@_{y}a \leq @_{y}a$,

\end{description}

\end{minipage}
\begin{minipage}{6cm}

\begin{description}

\item[($\mathit{ref}$)] $@_{x}x = \top$,

\item[($\mathit{introduction}$)] $x \wedge a \leq @_{x}a$, and

\item[($\mathit{back}$)] $\Diamond @_{x}a \leq @_{x}a$.

\end{description}

\end{minipage}
\end{center}

The class of hybrid $@$-algebras will be denoted by $\mathsf{H}@\mathsf{A}$.

A hybrid $@$-algebra $\frak{A} = (\A, X_A)$ is said to be \emph{permeated} if its hybrid algebra reduct is permeated.
The class of permeated hybrid $@$-algebras will be denoted by $\mathsf{PH}@\mathsf{A}$.
\end{definition}

As before, an \textit{assignment} on a hybrid $@$-algebra $\frak{A}$ is a map $v$: $\mathsf{PROP} \cup \mathsf{NOM} \to A$ associating an element of $A$ with each propositional variable in $\mathsf{PROP}$ and an atom of $X_A$ with each nominal in $\mathsf{NOM}$. Given such an assignment $v$, we calculate the \textit{meaning} $\tilde{v}(t)$ of a term $t$ as before. In particular, $\tilde{v}(@_{\nomj}\psi) = @_{\tilde{v}(\nomj)}\tilde{v}(\psi)$. Truth is defined as for hybrid algebras.


Proposition \ref{Lemma:@:BehavesCorrectly:H@A} below characterizes the behaviour of the $@$-operator in hybrid $@$-algebras.

\begin{prop} \label{Lemma:@:BehavesCorrectly:H@A}
Let $\frak{A} = (\mathbf{A}, X_A)$, where $\mathbf{A} = (A, \wedge, \vee, \neg, \bot, \top, \Diamond, @)$ with $(A, \wedge, \vee, \neg, \bot, \top, \Diamond)$ a BAO, $X_A$ is non-empty subset of atoms of $\mathbf{A}$, and $@$ is a binary operator with first coordinate ranging over $X_A$ and second coordinate over all elements of the algebra. Then $\frak{A}$ is a hybrid @-algebra iff for all $x \in X_A$ and $a \in A$, $@_{x}a = \top$ iff $x \leq a$ and $@_{x}a = \bot$ iff $x \nleq a$.
\end{prop}

\begin{proof}
The proof of the left-to-right direction is similar to that of Proposition \ref{Lemma:@:BehavesCorrectly:ortho}. For the converse direction, we have to show that $@$ satisfies $(K@$), $(\mathit{self}$-$\mathit{dual})$, $(\mathit{agree})$, $(\mathit{ref})$, $(\mathit{introduction})$ and $(\mathit{back})$.

For $(K@$), let $x \in X_A$ and $a, b \in A$, and assume $x \leq \neg a \vee b$. Then $@_x(\neg a \vee b) = \top$. But since $x$ is an atom, $x \leq \neg a$ or $x \leq b$. Hence, $x \nleq a$ or $x \leq b$, and so $@_xa = \bot$ or $@_xb = \top$. We then have $\neg @_xa = \top$ or $@_xb = \top$, which means that $\neg @_xa \vee @_xb = \top$. Therefore, $@_x(\neg a \vee b) \leq \neg @_xa \vee @_xb$, as required. On the other hand, if $x \nleq \neg a \vee b$, $@_x(\neg a \vee b) = \bot$, which gives $@_x(\neg a \vee b) \leq \neg @_xa \vee @_xb$.

For $(\mathit{self}$-$\mathit{dual})$, let $x \in X_A$ and $b \in A$, and assume $x \leq b$. Then $@_xb = \top$, and so $\neg @_xb = \bot$. But $x \leq b$ implies $x \nleq \neg b$, which means that $@_x\neg b = \bot$. Hence, $\neg @_xb = @_x\neg b$. On the other hand, if $x \nleq b$, then $@_xb = \bot$, which gives $\neg @_x b = \top$. But we know from $x \nleq b$ that $x \leq \neg b$, so $@_x\neg b = \top$.

To show that $@$ satisfies $(\mathit{agree})$, let $x, y \in X_A$ and $b \in A$, and assume $y \leq b$. Then $@_yb = \top$, and so $@_x@_yb \leq @_yb$. On the other hand, if $y \nleq b$, $@_yb = \bot$, which means that $@_x@_yb = @_x\bot = \bot$. Hence, $@_x@_yb \leq @_yb$.

For $(\mathit{ref})$, note that we have $x \leq x$ for all $x \in X_A$, so $@_xx = \top$.

$(\mathit{introduction})$ Let $x \in X$ and $b \in A$, and assume $x \leq b$. Then $@_xb = \top$, and so $x \wedge b \leq @_xb$. If $x \nleq b$, $x \wedge b = \bot$ and $@_{x}b = \bot$, which gives $x \wedge b \leq @_xb$.

Finally, to show that $@$ satisfies $(\mathit{back})$, let $x \in X_A$ and $b \in A$. If $x \leq b$, then $@_x b = \top$, which means that $\Diamond @_xb \leq @_xb$. On the other hand, assume $x \nleq b$. Then we have $@_xb = \bot$, and so we get $\Diamond @_xb = \Diamond\bot = \bot$. Therefore, $\Diamond @_xb \leq @_xb$.  \qedhere
\end{proof}

\begin{remark}
In an orthodox interpretation for $\mathcal{H}(@)$ it is not necessarily the case that $@_{s_{\nomi}}a = \bot$ iff $s_{\nomi} \nleq a$. This reflects the fact that the constant interpretations of nominals need not be atoms. Indeed, consider the orthodox interpretation $\mathbf{A} = (\mathbf{2}, \Diamond, @, \{s_{\nomj}\}_{\nomj \in \textsf{NOM}})$, where $\mathbf{2}$ is the two element Boolean algebra, $\Diamond 0 = 0$, $\Diamond 1 = 1$, $s_\nomj = 0$, and $s_\nomi = 0$ for $\nomi \neq \nomj$. Then $@_{s_{\nomj}}\neg s_{\nomj} = \neg @_{s_\nomj}s_{\nomj} = \neg 1 = 0$ but $s_{\nomj} = 0 \leq 1 = \neg s_{\nomj}$. This further motivates our choice to work with hybrid algebras instead of orthodox interpretations.
\end{remark}

\subsection{Duality with two sorted-general frames}

As already indicated, the duality between Boolean algebras with operators and descriptive general frames extends naturally to a duality between hybrid algebras and descriptive two-sorted general frames. In this subsection we will formulate these results more precisely. We will omit proofs, as these are relatively straightforward extensions of the corresponding proofs in the modal case which can be found e.g. in \cite[Section 5.4]{BdRV01}. For completely worked out proofs and also the duality of morphisms, we refer the interested reader to \cite[Chapter 2]{Robinson:PhD:Thesis}.

We turn two-sorted general frames into hybrid algebras, and vice versa, in a natural way. The underlying hybrid algebra of a two-sorted general frame consists of its algebra of admissible sets together with the singletons containing its admissible points:

\begin{definition}
Let $\frak{g} = (W, R, A, B)$ be a two-sorted general frame. The \emph{underlying hybrid algebra of} $\frak{g}$ is the structure
\[
\frak{g}^{*} = (A, \cap, \cup, -, \varnothing, W, \langle R\rangle, X_B),
\]
where $X_B = \{\{w\} \mid w \in B\}$.
\end{definition}

The two-sorted general ultrafilter frame of a hybrid algebra is  the (ordinary) ultrafilter frame of its BAO part, augmented with the set of principle ultrafilters generated by the designated atoms:

\begin{definition}
Let $\frak{A} = (\mathbf{A}, X_A)$ be a hybrid algebra. Then the \emph{two-sorted general ultrafilter frame of $\frak{A}$} is defined as
\[
\frak{A}_{*} = (\mathit{Uf}\frak{A}, Q_{\Diamond}, \widehat{A}, X_A\uparrow),
\]
where
\begin{itemize}
\item $\mathit{Uf}\frak{A}$ is the set of all ultrafilters of $\mathbf{A}$,
\item $Q_{\Diamond} \subseteq \mathit{Uf}\frak{A} \times \mathit{Uf}\frak{A}$ such that $Q_{\Diamond} u v$ iff $\Diamond a \in u$ for all $a \in v$,
\item $\widehat{A} := \{\widehat{a} \mid a \in A\}$ where $\widehat{a} = \{ u \in  \mathit{Uf}\frak{A} \mid a \in u\}$, and
\item $X_A\uparrow = \{x\uparrow \mid x \in X_A\}$ is the set of principle ultrafilters generated by the elements of $X_A$.
\end{itemize}
\end{definition}

These two constructions are systematically connected by the following theorem:

\begin{theorem}
Let $\frak{A} = (\mathbf{A}, X_A)$ be a hybrid algebra, and $\frak{g} = (W, R, A, B)$ a two-sorted general frame. Then
\begin{enumerate}

\item $\frak{A}_{*}$ is a descriptive two-sorted general frame,

\item $\frak{g}^{*}$ is a hybrid algebra,

\item $(\frak{A}_{*})^{*} \cong \frak{A}$,

\item $(\frak{g}^{*})_{*} \cong \frak{g}$ iff $\frak{g}$ is descriptive,

\item if $\frak{A}$ is permeated, then $\frak{A}_{*}$ is strongly descriptive, and

\item if $\frak{g}$ is strongly descriptive, then $\frak{g}^{*}$ is permeated.
\end{enumerate}
\end{theorem}

\section{Algebraic completeness} \label{sec:algebraic:comp}

We will now prove completeness of the axiomatizations in Section \ref{sec:prelim} with respect to the hybrid algebras introduced in Section \ref{sec:AlgbraicSem}. The general pattern is as follows: the axiomatizations without the additional `non-orthodox' rules are complete with respect to the class of hybrid algebras, whereas the axiomatizations with the additional `non-orthodox' rules are complete with respect to the class of permeated hybrid algebras.

\subsection{Algebraic completeness of $\mathbf{H}\oplus\Sigma$} \label{subsec:HSigma:complete}

The standard proof of the completeness of modal logics with respect to classes of BAOs proceeds by the well-known Lindenbaum-Tarski construction, see e.g. \cite{BdRV01}. In proving the completeness of logics $\mathbf{H}\oplus\Sigma$ with respect to classes of hybrid algebras we will also make use of this construction with, however, some significant complications. The main hurdle is the fact that the equivalence classes $[\nomi]$ of nominals need not be atoms of the Lindenbaum-Tarski algebra.
%
%
In order to be able to fall back on the established theory of modal logic,  we temporarily interpret the nominals as modal constants and work with the orthodox Lindenbaum-Tarski-algebra of $\mathbf{H}\oplus\Sigma$ over $\textsf{PROP}$. This algebra then requires a certain amount of sculpting to change it into a hybrid algebra of the right kind, as we will soon see.
We first state and prove the main theorem, and consequently prove the lemmas needed in it.

\begin{theorem} \label{Thm:HA:Comp}
For any set $\Sigma$ of $\mathcal{H}$-formulas, the logic $\mathbf{H}\oplus\Sigma$ is sound and complete with respect to the class of all hybrid algebras which validate $\Sigma$. That is, $\vdash_{\mathbf{H}\oplus\Sigma} \phi$ iff $\models_{\mathsf{HA}(\Sigma)} \phi \approx
\top$.
\end{theorem}

\begin{proof}
It is straightforward to check the soundness direction of the above. For the completeness direction, we prove the contrapositive. So suppose $\nvdash_{\mathbf{H}\oplus\Sigma} \phi$. We need to find a hybrid algebra $\frak{A}$ and an assignment $v$ such that $\frak{A}, v \not\models \phi \approx \top$. For the purpose of this proof, we will temporarily treat the nominals as modal constants and work with orthodox interpretations of $\mathcal{H}$.

Now, consider the orthodox Lindenbaum-Tarski algebra of $\textbf{H}\oplus\Sigma$ over $\mathsf{PROP}$, i.e., the usual Lindenbaum-Tarski algebra (see e.g. \cite{BdRV01}) of the logic $\textbf{H}\oplus\Sigma$ with the nominals treated like constants or $0$-ary modalities.  For the sake of brevity, we will denote this algebra simply by $\mathbf{A}$. Note that $[\neg\phi] > \bot$ in $\mathbf{A}$, for suppose not, then $\neg\phi$ must be provably equivalent to $\bot$, i.e. $\vdash_{\mathbf{H}\oplus\Sigma} \neg\phi \leftrightarrow \bot$. Hence, $\vdash_{\mathbf{H}\oplus\Sigma} \top \to \phi$ and therefore $\vdash_{\mathbf{H}\oplus\Sigma} \phi$, which is a contradiction.

The fact that $\mathbf{A}$ validates precisely the theorems of $\textbf{H}\oplus\Sigma$ is proved in the usual way. To see that $\mathbf{A} \not\models \phi \approx \top$, let $\nu$ be the map $\nu$: $\mathsf{PROP} \to A$ defined by $\nu(p) = [p]$. It can easily be verified by straightforward structural induction that $\widetilde{\nu}(\psi) =[\psi]$ for all formulas $\psi$ that use variables from the set $\mathsf{PROP}$. But then $\tilde{\nu}(\phi) = [\phi] \neq [\top] = \tilde{\nu}(\top)$, for otherwise, $[\phi] = [\top]$, which means that $[\neg\phi] = [\bot]$, a contradiction.

Next, consider the canonical extension $\mathbf{A}^{\delta}$ of $\mathbf{A}$ (as a BAO). First, note that since all axioms of $\textbf{H}$ are Sahlqvist under the orthodox interpretation, it follows from the canonicity of Sahlqvist equations that $\mathbf{A}^{\delta} \models \mathbf{H}^\approx$. However, the validity of the equations in $\Sigma^{\approx}$ is not necessarily preserved in passing from $\mathbf{A}$ to $\mathbf{A}^{\delta}$.

As mentioned at the end of Section \ref{sec:prelim}, $\Diamond$ and $\Box$ have right and left adjoints in $\mathbf{A}^{\delta}$, denoted by $\Box^{-1}$ and $\Diamond^{-1}$, respectively. 

Since $[\neg\phi] > \bot$ in $\mathbf{A}^{\delta}$ and $\mathbf{A}^{\delta}$ is atomic, there is some atom $d$ in $\mathbf{A}^{\delta}$ such that $d \leq [\neg\phi]$. Next define a sequence $d_0, d_1, d_2, \ldots$ of elements of $\mathbf{A}^{\delta}$, beginning with $d_0 = d$. Suppose $d_n$ is already defined, then let $d_{n + 1} = \Diamond^{-1}d_n$. Let
\[
D = \bigvee_{n \in \mathbb{N}}d_n.
\]
Define $\mathbf{A}_D = (A_D, \wedge^D, \vee^D, \neg^D, \bot^D, \top^D, \Diamond^D, \{s_{\nomj}^D\}_{\nomj \in \mathsf{NOM}})$,
where $A_D = \{a \wedge D \mid a \in A\}$, $\wedge^D$ and $\vee^D$ are the restriction of $\wedge$ and $\vee$ to $A_D$, and
\begin{align*}
\neg^D a &= \neg a \wedge D & \Diamond^D a &= \Diamond a \wedge D & s_{\nomj}^{D} &= s_{\nomj} \wedge D \\
\top^D &= D & \bot^D &= \bot
\end{align*}

Now, by Lemma \ref{Lemma:A_D:closed}, $\mathbf{A}_D$ is an algebra. We also know from Lemma \ref{Lemma:h:homomorphism} that $\mathbf{A}_D \models \textbf{H}\oplus\Sigma^{\approx}$. To see that $\mathbf{A}_D \not\models \phi \approx \top$, consider the assignment $\nu_D$: $\mathsf{PROP} \to A_D$ given by $\nu_D(p) = h(\nu(p))$, where $h$ is the homomorphism from $\mathbf{A}$ onto $\mathbf{A}_D$ defined in Lemma \ref{Lemma:h:homomorphism}. We can show by structural formula induction that $\tilde{\nu_D}(\psi) = h(\tilde{\nu}(\psi))$ for all $\mathcal{H}$-formulas $\psi$ that use variables from $\mathsf{PROP}$. Now, since $d \leq D$ and $d \leq [\neg\phi]$, $\tilde{\nu_D}(\neg\phi) = h(\tilde{\nu}(\neg\phi)) = \tilde{\nu}(\neg\phi) \wedge D = [\neg\phi] \wedge D \geq d > \bot$
Hence, $\tilde{\nu_D}(\neg\phi) \neq \bot^D$, and so $\tilde{\nu_D}(\phi) \neq D$, i.e., $\tilde{\nu_D}(\phi) \neq \tilde{\nu_D}(\top)$.

Next we want to produce a \emph{hybrid} algebra out of $\mathbf{A}_{D}$ which also refutes $\phi$. The desired result could be obtained if we could show that the constant interpretations of the nominals in $\mathbf{A}_{D}$ are atoms. This would allow us to drop the constants from the signature of $\mathbf{A}_D$ and replace them with a designated set of atoms consisting of all $s^D_{\nomi}$, $\nomi \in \mathsf{NOM}$. However, this is not necessarily the case. By Lemma \ref{s^_nomj:atoms:or:bot}, $s^D_{\nomi}$ is either an atom or $\bot$. So let $\frak{A}_{D} = (\mathbf{A}^-_{D}, X_{A_D})$, where $\mathbf{A}^-_{D}$ is the modal algebra reduct of $\mathbf{A}_{D}$ obtained by omitting the constant interpretations of nominals, and $X_{A_D} = \{s_{\nomi}^D \mid s_{\nomi}^D > \bot \}$. But this is still not necessarily a hybrid algebra since it is possible that $X_{A_D}$ can be empty. In fact, we have three possibilities, corresponding to the following three cases:

\paragraph{Case 1:} $s_{\nomi}^D > \bot$ for all $\nomi \in \mathsf{NOM}$. This is the simplest case. Since $X_{A_D} \neq \varnothing$, it follows from the foregoing that $\frak{A}_{D}$ is a hybrid algebra. Furthermore, since $\Sigma$ is closed under ($\mathit{Sorted}$ $\mathit{substitution}$), $\frak{A}_{D} \models \Sigma^{\approx}$. To see that $\frak{A}_D \not\models \phi \approx \top$, consider the assignment $\nu'_D$ which extends $\nu_D$ from $\mathsf{PROP}$ to $\mathsf{PROP} \cup \mathsf{NOM}$, obtained by simply setting $\nu'_D(\nomi) = s_{\nomi}^D$ for each $\nomi \in \mathsf{NOM}$. It is clear that $\widetilde{\nu}'_D(\psi) = \widetilde{\nu}_D(\psi)$ for all $\mathcal{H}$-formulas $\psi$, and hence $\widetilde{\nu}'_D(\phi) \neq \widetilde{\nu}'_D(\top)$.

\paragraph{Case 2:} $s_{\nomi}^D = \bot$ for some $\nomi \in \mathsf{NOM}$ but not all. From the foregoing, we know that $(\frak{A}_D)_0 \models \mathbf{H}\oplus\Sigma^{\approx}$, and hence, by Proposition \ref{Prop:Product:Pres}, $\mathfrak{A}_D \times \mathfrak{A}_D \models \textbf{H}\oplus\Sigma^{\approx}$. Now, let $\nomj \in \mathsf{NOM}$ such that $s_{\nomj}^D \neq \bot$.
%
Consider the assignment $\nu''_D$ obtained by setting $\nu''_D(p) = (\nu_D(p),\nu_D(p))$ for all propositional variables $p \in \mathsf{PROP}$, and
\[
\nu''_D(\nomi) = \left\{ \begin{array}{ll}
(s_{\nomi}^D, \bot) & \textnormal{ if } s_{\nomi}^{D} > \bot \\
(\bot, s_{\nomj}^D) & \textnormal{ if } s_{\nomi}^{D} = \bot
\end{array} \right.
\]
for all nominals $\nomi \in \textsf{NOM}$. It is straightforward to show (using structural induction) that for any $\mathcal{H}$-formula $\psi$, we have $\widetilde{\nu}''_D(\psi) = (\widetilde{\nu}_D(\psi), a_{\psi})$, where $a_{\psi}$ is some element of $\mathfrak{A}_D$. But then $\widetilde{\nu}''_D(\phi) = (\widetilde{\nu}_D (\phi), a_\phi) \neq (\widetilde{\nu}_D(\top), D) = \widetilde{\nu}''_D(\top)$ since $\widetilde{\nu}_D(\phi) \neq \widetilde{\nu}_D(\top)$.

\paragraph{Case 3:} $s_{\nomi}^D = \bot$ for all $\nomi \in \mathsf{NOM}$. In this case, $X_{A_D} = \varnothing$, so $\frak{A}_D$ is not a hybrid algebra. So we need another strategy for constructing a hybrid algebra that will work here. First, we claim that $[\nomi] > \bot$ in $\mathbf{A}$ for all $\nomi \in \mathsf{NOM}$. To see this, suppose $[\nomi] = \bot$. Then $\vdash \nomi \leftrightarrow \bot$, and so $\vdash \neg\nomi \leftrightarrow \top$. Hence, $\vdash \neg\nomi$, which means that $\vdash \bot$ by the (\textit{NameLite}) rule. However, this is a contradiction. We thus also have that $[\nomi] > \bot$ in $\mathbf{A}^{\delta}$ for all $\nomi \in \mathsf{NOM}$. So choose some nominal $\nomj \in \mathsf{NOM}$. Since $\mathbf{A}^{\delta}$ is atomic, there is an atom $d'$ in $\mathbf{A}^{\delta}$ such that $d' \leq [\nomj]$.
Now, using $d'$ instead of $d$, we define $D'$ and $\mathbf{A}_{D'}$ in the same way as $D$ and $\mathbf{A}_D$, i.e. by setting $d'_0 = d'$, $d'_{n+1} = \Diamond^{-1} d'_{n}$ and $D' = \bigvee_{n \in \mathbb{N}} d'_n$. In the same way as for $\mathbf{A}_D$, we can prove that $\mathbf{A}_{D'} \models \textbf{H}\oplus\Sigma^{\approx}$ and that $s_{\nomi}^{D'}$ is either $\bot$ or an atom of $\mathbf{A}_{D'}$. Define $\nu_{D'}$ in the same way as $\nu_D$. Note that we do not know if $\mathbf{A}_{D'}, \nu_{D'} \not\models \phi \approx \top$. But this is not a problem, as we will soon see. Now, let $\frak{A}_{D'} = (\mathbf{A}^-_{D'}, X_{A_{D'}})$ where $\mathbf{A}^-_{D'}$ is the reduct of $\mathbf{A}_{D'}$ obtained by omitting the constant interpretations of nominals and $X_{A_{D'}} = \{s_{\nomi}^{D'} \mid s_{\nomi}^{D'} > \bot \}$. We know that $X_{A_{D'}} \neq \varnothing$ since at least $s_{\nomj}^{D'} \neq \bot$. Furthermore, $(\frak{A}_{D'})_{0} \models \textbf{H}\oplus\Sigma^{\approx}$, and so, since $(\mathfrak{A}_D)_0 \models \textbf{H}\oplus\Sigma^{\approx}$, $\frak{A}_D \times \frak{A}_{D'} \models \textbf{H}\oplus\Sigma^{\approx}$ by Proposition \ref{Prop:Product:Pres}. To show that $\frak{A}_D \times \frak{A}_{D'} \not\models \phi \approx \top$, let $\nu'''$ be defined by $\nu'''_D(p) = (\nu_D(p),\nu_{D'}(p))$ for all propositional variables $p \in \mathsf{PROP}$, and
\[
\nu'''_D(\nomi) = \left\{ \begin{array}{ll}
(\bot, s_{\nomi}^{D'}) & \textnormal{ if } s_{\nomi}^{D'} > \bot \\
(\bot, s_{\nomj}^{D'}) & \textnormal{ if } s_{\nomi}^{D'} = \bot
\end{array} \right.
\]
for all nominals $\nomi \in \textsf{NOM}$. Using structural induction, we can show that for all $\mathcal{H}$-formulas $\psi$, $\widetilde{\nu}'''_{D}(\psi) = (\widetilde{\nu}_D(\psi), a'_\psi)$, where $a'_{\psi}$ is some element in $\frak{A}_{D'}$. But then $\widetilde{\nu}'''_D(\phi) = (\widetilde{\nu}_D(\phi), a'_\phi) \neq (\widetilde{\nu}_D(\top), D') = \widetilde{\nu}'''_D(\top)$ since $\widetilde{\nu}_D(\phi) \neq \widetilde{\nu}_D(\top)$.	
\end{proof}

We will now prove the lemmas used in the proof of the above theorem. Unless stated otherwise, in what follows $\mathbf{A}$, $\mathbf{A}^{\delta}$, $\mathbf{A_D}$, $\nu$ and $\nu_D$ will be as in the proof of Theorem \ref{Thm:HA:Comp}. The first lemma we need is that $\mathbf{A}_D$ is an algebra. To prove this, we have to show that $A_D$ is closed under the operations defined in the proof of Theorem \ref{Thm:HA:Comp}. But first we need the following lemma:

\begin{lemma}\label{Lemma:DleqBoxD}
$D \leq \Box D$
\end{lemma}

\begin{proof}
By the definition of $D$, $\Box D = \Box \bigvee_{n \in \mathbb{N}} d_n$, and so, since $\Box$ is monotone, $\Box D \geq \bigvee_{n \in \mathbb{N}} \Box d_n$. But by the definition of $d_n$, $\bigvee_{n \in \mathbb{N}} \Box d_n = \left(\bigvee_{n \in \mathbb{N} - \{0\}}\Box\Diamond^{-1}d_{n - 1}\right) \vee \Box d_0$, so $\Box D \geq \left(\bigvee_{n \in \mathbb{N} - \{0\}}\Box\Diamond^{-1}d_{n - 1}\right) \vee \Box d_0$. Hence, $\Box D \geq \left(\bigvee_{n \in \mathbb{N} - \{0\}} \Box  \Diamond^{-1} d_{n - 1}\right)$. Since $\Box$ and $\Diamond^{-1}$ are adjoint, we have $\Box \Diamond^{-1} d_i \geq d_i$, and therefore $\Box D \geq  \bigvee_{n \in \mathbb{N} - \{0\}} d_{n - 1}$. We thus have that $\Box D \geq D$ by the definition of $D$.
\end{proof}

\begin{lemma}\label{Lemma:A_D:closed}
$A_D$ is closed under the operations $\wedge^D$, $\vee^D$, $\neg^D$, and $\Diamond^D$.
\end{lemma}

\begin{proof}
The cases for $\wedge^D$ and $\vee^D$ are straightforward, so we consider the cases for $\neg^D$ and $\Diamond^{D}$. Let $a \in A_{D}$. Then $a = a' \wedge D$ for some $a' \in A$. Firstly, $\neg^Da = \neg(a' \wedge D) \wedge D = (\neg a' \vee \neg D) \wedge D = (\neg a' \wedge D) \vee \bot = \neg a' \wedge D$. But $A$ is closed under $\neg$, so $\neg a' \in A$, which means that $\neg^D a \in A_D$.

Next, $\Diamond^Da = \Diamond(a' \wedge D) \wedge D \leq \Diamond a' \wedge D$,
and conversely, $\Diamond^Da = \Diamond(a' \wedge D) \wedge D \geq \Diamond a' \wedge \Box D \wedge D = \Diamond a' \wedge D$,
where the last step follows from Lemma \ref{Lemma:DleqBoxD}. Hence, $\Diamond^Da = \Diamond a' \wedge D$, and so, since $A$ is closed under $\Diamond$, $\Diamond a' \in A$. We thus have that $\Diamond^Da \in A_{D}$.
\end{proof}

To show that the algebra $\mathbf{A}_{D}$ validates the equations in $\mathbf{H}\oplus\Sigma^{\approx}$, we prove that $\mathbf{A}_{D}$ is a homomorphic image of $\mathbf{A}$.

\begin{lemma}\label{Lemma:h:homomorphism}
The map $h$: $A \to A_D$ defined by $h(a) = a \wedge D$ is a surjective homomorphism from $\mathbf{A}$ onto $\mathbf{A}_D$.
\end{lemma}

\begin{proof}
First, $h$ is clearly surjective. In verifying that $h$ is a homomorphism, all cases except those for $\neg$ and $\Diamond$ are straightforward. The case for $\neg$ is proved as follows:
$h(\neg a) = \neg a \wedge D = (\neg a \wedge D) \vee (\neg D \wedge D) = (\neg a \vee \neg D) \wedge D = \neg(a \wedge D) \wedge D = \neg^{D}h(a)$. The right-to-left inequality for $\Diamond$ is proved as follows:
$h(\Diamond a) = \Diamond a \wedge D \geq \Diamond(a \wedge D) \wedge D = \Diamond h(a) \wedge D = \Diamond^{D}h(a)$,
where the inequality follows from the monotonicity of $\Diamond$. Conversely,
$\Diamond^{D}h(a)  = \Diamond(a \wedge D) \wedge D \geq \Diamond a \wedge \Box D \wedge D = \Diamond a \wedge D = h(\Diamond a)$.
Here the third step follows from Lemma \ref{Lemma:DleqBoxD}.
\end{proof}

Our next goal is to show that the constant interpretation of a nominal is either an atom of $\mathbf{A}_D$ or $\bot$. But first, we need the following result:

\begin{lemma}\label{alD^-1sbiffblDsa: alDisblDsa}
Let $\mathbf{A}$ be a BAO, and let $a$ and $b$ be atoms of the canonical extension $\mathbf{A}^{\delta}$ of $\mathbf{A}$. Then $a \leq (\Diamond^{-1})^{m}b$ iff $b \leq \Diamond^{m}a$.
\end{lemma}

\begin{proof}
For the left-to-right implication, let $a$ and $b$ be two atoms in $\mathbf{A}^{\delta}$, and assume \break $a \leq (\Diamond^{-1})^{m}b$. Suppose that $b \nleq \Diamond^{m}a$. Then, since $b$ is an atom, $b \leq \neg\Diamond^{m}a$, and so $\Diamond^{m}a \leq \neg b$. Now applying adjunction $m$ times yields $a \leq (\Box^{-1})^m \neg b$,
which means that $(\Diamond^{-1})^{m}b$ $\leq \neg a$. Hence, by our assumption, $a \leq \neg a$, and so $\bot = a \wedge \neg a = a$, which is a contradiction.

The converse implication is similar.
\end{proof}

\begin{lemma} \label{s^_nomj:atoms:or:bot}
For each $\nomi \in \mathsf{NOM}$, $s_{\nomi}^D$ is either $\bot$ or an atom of $\mathbf{A}^{\delta}$, and hence, of $\mathbf{A}_{D}$.
\end{lemma}
\begin{proof}
Suppose that $s_{\nomi}^D \neq \bot$. We show that if $a, b \in At\mathbf{A}^{\delta}$ such that $a, b \leq s_{\nomi}^{D}$, then $a = b$. So let $a, b \in At \mathbf{A}^{\delta}$ such that $a, b \leq s_{\nomi}^{D}$. Then $a, b \leq D$, and so $a \leq (\Diamond^{-1})^{n_{1}}d$ and $b \leq (\Diamond^{-1})^{n_{2}}d$ for some natural numbers $n_1$ and $n_2$. We thus have that $d \leq \Diamond^{n_{1}}a$ and $d \leq \Diamond^{n_{2}}b$ by Lemma \ref{alD^-1sbiffblDsa: alDisblDsa}, and so, since $a, b \leq s_{\nomi}^D$ we find $d \leq \Diamond^{n_{1}}(a \wedge s_{\nomi}^d)$ and $d \leq \Diamond^{n_{2}}(b \wedge s_{\nomi}^D)$. Hence, $d \leq \Diamond^{n_{1}}(a \wedge s_{\nomi} \wedge D) \leq \Diamond^{n_{1}}(a \wedge s_{\nomi})$ and $d \leq \Diamond^{n_{2}}(b \wedge s_{\nomi} \wedge D) \leq \Diamond^{n_{2}}(b \wedge s_{\nomi})$. So, by axiom (Nom), $d \leq \Box^{m}(\neg s_{\nomi} \vee a)$ and $d \leq \Box^{m}(\neg s_{\nomi} \vee b)$ for all natural numbers $m$. Since $\Diamond^{-1}$ and $\Box$ are adjoint, $(\Diamond^{-1})^{m}d \leq \neg s_{\nomi} \vee a$ and $(\Diamond^{-1})^{m}d \leq \neg s_{\nomi} \vee b$, which means that, for all $m \in \mathbb{N}$, it holds that $(\Diamond^{-1})^{m}d \leq (\neg s_{\nomi} \vee a) \wedge (\neg s_{\nomi} \vee b)$. It thus follows that $D \leq (\neg s_{\nomi} \vee a) \wedge (\neg s_{\nomi} \vee b) = \neg s_{\nomi} \vee (a \wedge b)$. Now, if $a \neq b$, $D \leq \neg s_{\nomi}$, so $s_{\nomi} \leq \neg D$. Hence, $s_{\nomi} \wedge D \leq \bot$, i.e., $s_{\nomi}^D = \bot$, contradicting our assumption that $s_{\nomi}^D \neq \bot$.
\end{proof}

\subsection{Algebraic completeness of $\mathbf{H}^{+}\oplus\Sigma$}

As we will now show, the logic $\mathbf{H}^{+}\oplus\Sigma$ is complete with respect to permeated hybrid algebras which validate $\Sigma$. We first state and prove the main theorem, and consequently prove the lemmas needed for the main proof.

\begin{theorem} \label{Thm:PHA:Comp}
For any set $\Sigma$ of $\mathcal{H}$-formulas, the logic $\mathbf{H}^{+}\oplus\Sigma$ is sound and complete with respect to the class of all permeated hybrid algebras which validate $\Sigma$. That is to say, $\vdash_{\mathbf{H}^{+}\oplus\Sigma} \phi$ iff $\models_{\mathsf{PHA}(\Sigma)} \phi \approx \top$.
\end{theorem}

\begin{proof}
Suppose $\nvdash_{\mathbf{H}^{+}\oplus\Sigma} \phi$. We have to find a permeated hybrid algebra $\frak{A}$ and an assignment $v$ such that $\frak{A}, v \not\models \phi \approx \top$. However, as before, we will work with the orthodox interpretation of $\mathcal{H}$ for the purpose of the proof. Let $\mathsf{NOM}'$ be a denumerably infinite set of nominals disjoint from $\mathsf{NOM}$. We know from \cite{TenCate:Phd:Thesis} that $\neg \phi$ is contained in a $\mathbf{H}^{+}\oplus\Sigma$-maximal consistent set of formulas $\Gamma$ in the extended language such that
\begin{enumerate}[(i)]
\item $\Gamma$ contains at least one nominal, say $\nomi_0$, and
\item for each formula of the form $\Diamond^n(\nomi \wedge \Diamond \phi)$ in $\Gamma$, there is a nominal $\nomj$ for which the formula $\Diamond^n(\nomi \wedge \Diamond (\nomj \wedge \phi))$ is in $\Gamma$.
\end{enumerate}

Now, consider the orthodox Lindenbaum-Tarski algebra of $\mathbf{H}^{+}\oplus\Sigma$ over $\mathsf{PROP}$. For simplicity, denote this algebra by $\mathbf{A}$. In the usual way, we can show that $\mathbf{A} \models \mathbf{H}^{+}\oplus\Sigma^{\approx}$ and $\mathbf{A}, \nu \not\models \phi \approx \top$, where $\nu$ is the natural map taking $p$ to $[p]$. By standard propositional reasoning, the set $[\Gamma] = \{[\gamma] \mid \gamma \in \Gamma \}$ is an ultrafilter of $\mathbf{A}$. We can thus use the finite meet property to conclude that for every finite subset $\Gamma' \subseteq \Gamma$, it holds that $\bigwedge [\Gamma'] > \bot$ in $\mathbf{A}$.

Next we consider the orthodox canonical extension $\mathbf{A}^{\delta}$ of $\mathbf{A}$. Note that in $\mathbf{A}^{\delta}$ we have $\bigwedge [\Gamma] > \bot$. To see this, suppose $\bigwedge [\Gamma] \leq \bot$. But
by the compactness of the embedding of $\mathbf{A}$ into $\mathbf{A}^{\delta}$, there is a finite subset $\Gamma' \subseteq \Gamma$ such that $\bigwedge [\Gamma'] \leq \bot$ in $\mathbf{A}$, contradicting the claim above.

Since $\mathbf{A}^{\delta}$ is atomic, there is some atom $d$ in $\mathbf{A}^{\delta}$ such that $d \leq \bigwedge [\Gamma]$. Define, as in the proof of Theorem \ref{Thm:HA:Comp}, a sequence $d_0, d_1, d_2, \ldots$ of elements of $\mathbf{A}^{\delta}$, by setting $d_0 = d$ and $d_{n + 1} = \Diamond^{-1}d_n$. Set $D = \bigvee_{n \in \mathbb{N}}d_n$. Also define $\mathbf{A}_{D}$ as in the proof of Theorem \ref{Thm:HA:Comp}. As before it can be checked that $\mathbf{A}_D$, so defined, is an algebra. Furthermore, using the map $h$: $A \to A_D$ defined by $h(a) = a \wedge D$, we can show in the same way as in the proof of Lemma \ref{Lemma:h:homomorphism} that $\mathbf{A}_D$ is a homomorphic image of $\mathbf{A}$. Hence, $\mathbf{A}_D \models \mathbf{H}^{+}\oplus\Sigma^{\approx}$. To see that $\mathbf{A}_D \not\models \phi \approx \top$, consider the assignment $\nu_D(p)$: $\textsf{PROP} \to A_D$ defined as in the proof of Theorem \ref{Thm:HA:Comp}. Using the fact that $h$ is a homomorphism, we can show by structural induction that $\widetilde{\nu}_D(\psi) = h(\widetilde{\nu}(\psi))$ for all $\mathcal{H}$-formulas $\psi$. Now, since $\bigwedge [\Gamma] \leq \bigwedge [\Gamma']$ for every finite subset $\Gamma' \subseteq \Gamma$, we have $\nu(\bigwedge \Gamma') \geq d$. It thus holds that
$\widetilde{\nu}_D(\bigwedge \Gamma') = h(\widetilde{\nu}(\bigwedge \Gamma')) = \widetilde{\nu}(\bigwedge \Gamma') \wedge D \geq d > \bot$
for every finite subset $\Gamma'$ of $\Gamma$. But $\{\neg\phi\}$ is a finite subset of $\Gamma$, so $\widetilde{\nu}_D(\phi) \neq \top$.

In the same way as in the proof of Lemma \ref{s^_nomj:atoms:or:bot}, we can show that for each $\nomi \in \mathsf{NOM} \cup \mathsf{NOM}'$, $s^D_{\nomi}$ is either $\bot$ or an atom of $\mathbf{A}_D$. So let $\frak{A}_{D} = (\mathbf{A}^-_{D}, X_{A_D})$, where $\mathbf{A}^-_{D}$ is the reduct of $\mathbf{A}_{D}$ obtained by omitting the constant interpretations of nominals and $X_{A_D} = \{s_{\nomi}^D \mid s_{\nomi}^D > \bot \}$. We claim that $X_{A_D} \neq \varnothing$. To see this, recall that $\nomi_0 \in \Gamma$, so $d \leq \bigwedge[\Gamma] \leq [\nomi_0]$. It follows that $d \leq s^D_{\nomi_0}$, and so, at least $s^D_{\nomi_0} > \bot$. In particular, since $s^D_{\nomi_0}$ and $d$ are both atoms, $s^D_{\nomi_0} = d$. Now, from the fact that $\mathbf{H}^{+}\oplus\Sigma$ is closed under $(\mathit{Sorted}$ $\mathit{substitution})$, we have $\frak{A}_{D} \models \mathbf{H}^{+}\oplus\Sigma^{\approx}$. Furthermore, by Lemma \ref{Lemma:A_D:permeated}, $\frak{A}_D$ is permeated. From here we split our reasoning into two cases depending on whether the constant interpretations of all nominals in $\mathbf{A}_{D}$ are atoms or not:

\paragraph{Case 1:} $s_{\nomi}^D > \bot$ for all $\nomi \in \mathsf{NOM} \cup \mathsf{NOM}'$. In this case, consider the assignment $\nu'_D$ which extends $\nu_D$ from $\mathsf{PROP}$ to $\mathsf{PROP} \cup \mathsf{NOM} \cup \mathsf{NOM}'$, obtained by simply setting $\nu'_D(\nomi) = s_{\nomi}^D$ for each $\nomi \in \mathsf{NOM} \cup \mathsf{NOM}'$. It is clear that $\widetilde{\nu}'_D(\psi) = \widetilde{\nu}_D(\psi)$ for all $\mathcal{H}$-formulas $\psi$, and hence, $\widetilde{\nu}'_D(\phi) \neq \widetilde{\nu}'_D(\top)$. This means that $\frak{A}_{D}, \nu'_D \not\models \phi \approx \top$, as required.

\paragraph{Case 2:} $s_{\nomi}^D = \bot$ for some $\nomi \in \mathsf{NOM} \cup \mathsf{NOM}'$. Here we have that $(\mathfrak{A}_D)_0 \models \mathbf{H}^{+}\oplus\Sigma^{\approx}$, and hence, by Proposition \ref{Prop:Product:Pres}, $\frak{A}_D \times \frak{A}_D \models \mathbf{H}^{+}\oplus\Sigma^{\approx}$. By Proposition \ref{Prop:Perm:preserved:prod}, we know that $\frak{A}_D \times \frak{A}_D$ is also permeated. Now, consider the assignment $\nu''_D$ obtained by setting $\nu''_D(p) = (\nu_D(p), \nu_D(p))$ for all propositional variables $p \in \mathsf{PROP}$, and
\[
\nu''_D(\nomj) = \left\{ \begin{array}{ll}
(s_{\nomj}^{D}, \bot) & \textnormal{ if } s_{\nomj}^{D} > \bot \\
(\bot, s_{\nomi_0}^{D}) & \textnormal{ if } s_{\nomj}^{D} = \bot
\end{array} \right.
\]
for all nominals $\nomj \in \mathsf{NOM} \cup \mathsf{NOM}'$. For any $\mathcal{H}$-formula $\psi$, we have $\widetilde{\nu}''_D(\psi) = (\widetilde{\nu}_D(\psi), a_{\psi})$, where $a_{\psi}$ is some element of $\mathfrak{A}_D$. But then $\widetilde{\nu}''_D(\phi) = (\widetilde{\nu}_D(\phi), a_{\phi}) \neq (\widetilde{\nu}_D(\top), D) = \widetilde{\nu}''_D(\top)$.
\end{proof}

Let us now prove the lemmas used in the proof of Theorem \ref{Thm:PHA:Comp}. In what follows, unless stated otherwise, $\Gamma$, $\mathbf{A}$, $\mathbf{A}^{\delta}$, $\mathbf{A}_D$, $\nu$, $\nu_D$ and $\frak{A}_D$ will be as in the proof of Theorem \ref{Thm:PHA:Comp}.


\begin{lemma}\label{inverse>bot:dleqdiamomd}
Let $\mathbf{A}$ be a BAO, $\mathbf{A}^{\delta}$ its canonical extension, and $a, b\in \mathbf{A}^{\delta}$. Then we have $a \wedge (\Diamond^{-1})^nb > \bot$ iff $b \leq \Diamond^na$.
\end{lemma}

\begin{proof}
For the left-to-right implication, assume $a \wedge (\Diamond^{-1})^nb > \bot$. Then there is a $c \in At\mathbf{A}^{\delta}$ such that $c \leq a \wedge (\Diamond^{-1})^nb$, and so $c \leq a$ and $c \leq (\Diamond^{-1})^nb$. Hence, by Lemma \ref{alD^-1sbiffblDsa: alDisblDsa}, $b \leq \Diamond^{n}c$, which means that $b \leq \Diamond^n a$.

For the converse, suppose $a \wedge (\Diamond^{-1})^nb = \bot$. Then $(\Diamond^{-1})^nb \leq \neg a$, and so, since $\Diamond^{-1}$ and $\Box$ are adjoint, $b \leq \Box^n\neg a = \neg\Diamond^n a$. Hence, $b \nleq \Diamond^n a$.
\end{proof}

\begin{lemma} \label{Lemma:DiamondBox:leq:Diamond}
Let $\mathbf{A}$ be a BAO and $a, b$ elements of $\mathbf{A}$. Then $\Box^ma \wedge \Diamond^m \leq \Diamond^m(a \wedge b)$.
\end{lemma}

\begin{proof}
We know that $\Box^m(\neg a \vee \neg b) = \neg\Box^m a \vee \Box^m\neg b$, so $\neg (\neg\Box^m a \vee \Box^m\neg b) \leq \neg\Box^m(\neg a \vee \neg b)$. But $\neg (\neg\Box^m a \vee \Box^m\neg b) = \Box^ma \wedge \neg\Box^m \neg b = \Box^ma \wedge \Diamond^m b$, which means that $\Box^ma \wedge \Diamond^m b \leq \neg\Box^m(\neg a \vee \neg b)$. Hence, since $\neg\Box^m(\neg a \vee \neg b) = \Diamond^m(a \wedge b)$, we find $\Box^ma \wedge \Diamond^m b \leq \Diamond^m(a \wedge b)$.
\end{proof}

\begin{lemma}\label{dleqdiamond:dleqdiamondd}
For any element $a$ in $\mathbf{A}_{D}$, if $d \leq \Diamond^na$, then $d \leq (\Diamond^D)^na$.
\end{lemma}

\begin{proof}
The proof is by induction on $n$. For $n = 1$, assume $d \leq \Diamond a$. Then $d \wedge D \leq \Diamond a \wedge D$, and so, since $d \leq D$, we have $d \leq \Diamond^D a$. Now, suppose that for all $a$ in $\mathbf{A}_D$, the claim holds for $n = k$. For $n = k + 1$, assume $d \leq \Diamond^{k + 1}a$. Then $d \leq \Diamond^k\Diamond a$. But by Lemma \ref{Lemma:DleqBoxD}, $d \leq D \leq \Box^kD$, so using Lemma \ref{Lemma:DiamondBox:leq:Diamond} we find that $d \leq \Diamond^{k}\Diamond a \wedge \Box^kD \leq \Diamond^k(\Diamond a \wedge D) = \Diamond^k\Diamond^Da$. Now, since $\Diamond^D a \in A_D$, we can use the inductive hypothesis to get $d \leq (\Diamond^D)^{k}\Diamond^{D} a = (\Diamond^D)^{k+1} a$.
%
\end{proof}

\begin{lemma} \label{Lemma:gammainGamma:dleq}
For all $\mathcal{H}$-formulas $\gamma$, it holds that $\gamma \in \Gamma$ iff $d \leq \widetilde{\nu}(\gamma)$ iff $d \leq \widetilde{\nu}_D(\gamma)$.
\end{lemma}

\begin{proof}
Suppose $\gamma \in \Gamma$. Then $[\gamma] \in [\Gamma]$, and so $\bigwedge[\Gamma] \leq [\gamma]$. But $d \leq \bigwedge[\Gamma]$, so $d \leq [\gamma] = \widetilde{\nu}(\gamma)$. Conversely, assume $\gamma \notin \Gamma$. Since $\Gamma$ is a maximal consistent set of formulas, we have $\neg\gamma \in \Gamma$. Hence, using the left-to-right direction, $d \leq \widetilde{\nu}(\neg\gamma) = \neg\widetilde{\nu}(\gamma)$, and so $d \nleq \widetilde{\nu}(\gamma)$. Now, assume $d \leq \widetilde{\nu}(\gamma)$. Then $d \wedge D \leq \widetilde{\nu}(\gamma) \wedge D$, and so $d \leq \widetilde{\nu}_{D}(\gamma)$. Conversely, assume $d \leq \widetilde{\nu}_{D}(\gamma)$. Then $d \leq \widetilde{\nu}(\gamma) \wedge D$, which gives $d \leq \widetilde{\nu}(\gamma)$.
\end{proof}

\begin{lemma} \label{Lemma:leastm}
Let $a$ be any element in $A_D$ and $n$ a natural number such that $a \wedge (\Diamond^{-1})^{n}d > \bot$. Then there is an $s^D_{\nomj}$ such that $a \wedge s^D_{\nomj} \wedge (\Diamond^{-1})^{n}d > \bot$.
\end{lemma}

\begin{proof}
Let $a$ be any element of $A_D$ such that $a \wedge (\Diamond^{-1})^{n}d > \bot$. The proof is by induction on $n$. For $n = 0$, we have $a \wedge d > \bot$. But $d = s^D_{\nomi_{0}}$, so we are done. Now, suppose that for every $a \in A_D$, the claim holds for all $n = k$. For $n = k + 1$, assume that $a \wedge (\Diamond^{-1})^{k + 1}d > \bot$. Then $d \leq \Diamond^{k + 1}a = \Diamond^k\Diamond a$ by Lemma \ref{inverse>bot:dleqdiamomd}, and therefore, another application of Lemma \ref{inverse>bot:dleqdiamomd}, but in the opposite direction, yields $\Diamond a \wedge (\Diamond^{-1})^{k}d > \bot$. Hence, since $(\Diamond^{-1})^kd \leq D$, we have $\Diamond a \wedge ((\Diamond^{-1})^{k}d \wedge D) = (\Diamond a \wedge D) \wedge (\Diamond^{-1})^{k}d = \Diamond^{D} a \wedge (\Diamond^{-1})^{k}d > \bot$. Since $\Diamond^Da$ is an element of $\mathbf{A}_D$, we can use the inductive hypothesis to conclude that there is a nominal $\nomj \in \mathsf{NOM} \cup \mathsf{NOM}'$ such that $\Diamond^{D} a \wedge s^{D}_{\nomj} \wedge (\Diamond^{-1})^{k}d > \bot$. But then $d \leq \Diamond^{k}(\Diamond^D a \wedge s^{D}_{\nomj})$ by Lemma \ref{inverse>bot:dleqdiamomd}. Therefore, by Lemma \ref{dleqdiamond:dleqdiamondd}, $d \leq (\Diamond^{D})^{k}(\Diamond^D a \wedge s^{D}_{\nomj})$. Now, since $\nu$ and $h$ are surjective, there is some $\psi$ such that $d \leq (\Diamond^{D})^{k}(\Diamond^{D}\widetilde{\nu}_{D}(\psi) \wedge s^{D}_{\nomj}) = \widetilde{\nu}_{D}(\Diamond^{k}(\Diamond\psi \wedge \nomj))$.
We thus have that $\Diamond^{k}(\nomj \wedge \Diamond\psi) \in \Gamma$ by Lemma \ref{Lemma:gammainGamma:dleq}, which means there is a nominal $\nomk$ in $\mathsf{NOM} \cup \mathsf{NOM}'$ such that $\Diamond^{k}(\nomj \wedge \Diamond(\nomk \wedge \psi)) \in \Gamma$. Hence, by Lemma \ref{Lemma:gammainGamma:dleq}, we get $d \leq \widetilde{\nu}_{D}(\Diamond^{k}(\nomj \wedge \Diamond(\nomk \wedge \psi))) = (\Diamond^{D})^{k}(s^{D}_{\nomj} \wedge \Diamond^D(s^{D}_{\nomk} \wedge a)) = \Diamond^{k}(s^{D}_{\nomj} \wedge \Diamond(s^{D}_{\nomk} \wedge a)) \wedge D \leq \Diamond^{k}(s^{D}_{\nomj} \wedge \Diamond(s^{D}_{\nomk} \wedge a))$. So, by Lemma \ref{inverse>bot:dleqdiamomd}, $\bot < (\Diamond^{-1})^{k}d \wedge s^{D}_{\nomj} \wedge \Diamond(s^{D}_{\nomk} \wedge a) \leq (\Diamond^{-1})^{k}d \wedge \Diamond(s^{D}_{\nomk} \wedge a)$. But then $d \leq \Diamond^{k}\Diamond(s^{D}_{\nomk} \wedge a) = \Diamond^{k + 1}(s^{D}_{\nomk} \wedge a)$, and therefore, $(\Diamond^{-1})^{k + 1}d \wedge s^{D}_{\nomk} \wedge a > \bot$.
\end{proof}

Finally, we are ready to show that $\frak{A}_D$ is permeated.

\begin{lemma} \label{Lemma:A_D:permeated}
$\frak{A}_{D}$ is permeated.
\end{lemma}

\begin{proof}
For the first condition, let $b \in A_{D}$ such that $b > \bot$. Then $b \wedge (\Diamond^{-1})^{m}d > \bot$ for some $m \in \mathbb{N}$. By Lemma \ref{Lemma:leastm}, there is an $s^D_{\nomj}$ such that $b \wedge s^D_{\nomj} \wedge (\Diamond^{-1})^{m}d > \bot$. Then $d \leq \Diamond^{m}(b \wedge s^D_{\nomj})$ by Lemma \ref{inverse>bot:dleqdiamomd}. Now, $s^D_{\nomj} \neq \bot$, for otherwise, $d \leq \Diamond^{m(b)}\bot = \bot$, a contradiction. Hence, $s^D_{\nomj}$ is an atom by Lemma \ref{s^_nomj:atoms:or:bot}, and furthermore, $s^D_{\nomj} \in X_{A_D}$. We also claim that $s^D_{\nomj} \leq b$, for if not, $b \wedge s^D_{\nomj} = \bot$, giving $d \leq \bot$, a contradiction.

To prove the second condition, let $b \in A_D$ and $s^D_{\nomi} \in X_{A_D}$, and assume that $s^{D}_{\nomi} \leq \Diamond^{D}b$. Then $s^{D}_{\nomi} \leq \Diamond b \wedge D$, which means that $s^{D}_{\nomi} \leq \Diamond b \wedge (\Diamond^{-1})^{m}d$ for some $m \in \mathbb{N}$. Hence, $s^{D}_{\nomi} \wedge \Diamond b \wedge (\Diamond^{-1})^{m}d = s^{D}_{\nomi} > \bot$, and so by Lemma \ref{inverse>bot:dleqdiamomd} $d \leq \Diamond^{m}(s^{D}_{\nomi} \wedge \Diamond b)$. We thus have $d \leq (\Diamond^{D})^{m}(s^{D}_{\nomi} \wedge \Diamond b)$ by Lemma \ref{dleqdiamond:dleqdiamondd}, and therefore, since $d \leq D \leq (\Box^D)^mD$, $d \leq (\Diamond^{D})^{m}(s^{D}_{\nomi} \wedge \Diamond b) \wedge (\Box^{D})^{m}D$. But since $\mathbf{A}_D$ validates the axioms, we can show in the same way as in Lemma \ref{Lemma:DiamondBox:leq:Diamond} that for all $a, b \in A_D$ and all $n \in \mathbb{N}$, $(\Box^D)^na \wedge (\Diamond^D)^ \leq (\Diamond)^D)^n(a \wedge b)$, and thus $d \leq (\Diamond^{D})^{m}(s^{D}_{\nomi} \wedge \Diamond b \wedge D) = (\Diamond^{D})^{m}(s^{D}_{\nomi} \wedge \Diamond^{D} b)$. Now, since $\nu$ and $h$ are both surjective, there is some $\psi$ such that $d \leq (\Diamond^{D})^{m}(s^{D}_{\nomi} \wedge \Diamond^{D}\widetilde{\nu}_{D}(\psi)) = \widetilde{\nu}_{D}(\Diamond^{m}(\nomi \wedge \Diamond\psi))$. Hence, by Lemma \ref{Lemma:gammainGamma:dleq}, $\Diamond^{m}(\nomi \wedge \Diamond\psi) \in \Gamma$. This means that there is a $\nomj \in \mathsf{NOM} \cup \mathsf{NOM}'$ such that $\Diamond^{m}(\nomi \wedge \Diamond(\nomj \wedge \psi)) \in \Gamma$. Another application of Lemma \ref{Lemma:gammainGamma:dleq}, but in the opposite direction, gives $d \leq \widetilde{\nu}_{D}(\Diamond^{m}(\nomi \wedge \Diamond(\nomj \wedge \psi))) = (\Diamond^{D})^{m}(s^{D}_{\nomi} \wedge \Diamond^{D}(s^{D}_{\nomj} \wedge b))$.

Now, first note that $s^D_{\nomj} \neq \bot$, for else $d = \bot$, which is a contradiction. This means that $s^D_{\nomj}$ is an atom and in $X_{A_D}$. Second, $s^{D}_{\nomj} \leq b$, for otherwise, $s^{D}_{\nomj} \wedge b = \bot$, giving $d = \bot$, again a contradiction. Finally, to see that $s^{D}_{\nomi} \leq \Diamond^Ds^D_{\nomj}$, suppose for the sake of a contradiction that it is not. Then $s^{D}_{\nomi} \wedge \Diamond^Ds^D_{\nomj} = \bot$, and so we have that $d \leq (\Diamond^D)^{m}(s^{D}_{\nomi} \wedge \Diamond^D(s^{D}_{\nomj} \wedge b))
\leq (\Diamond^D)^{m}(s^{D}_{\nomi} \wedge \Diamond^D s^{D}_{\nomj}) = (\Diamond^D)^{m}\bot = \bot$,
which is a contradiction.
\end{proof}

\subsubsection{Algebraic completeness of $\mathbf{H}(@)\oplus\Sigma$}

We now turn our attention to the language and logics with the satisfaction operator $@$. In broad strokes, the strategy for proving the completeness of axiomatic extensions of $\mathbf{H}(@)$ with respect to hybrid $@$-algebras is similar to that for axiomatic extensions of $\mathbf{H}$ and hybrid algebras. There are, however, a few significant differences: firstly, in defining the element $D$ from which the algebra $\mathbf{A}_D$ is constructed, instead of closing a singleton set under $\Diamond^{-1}$, we need to include the interpretations of all relevant nominals. Via duality, this can be seen as analogous to including all states named by nominals when forming generated submodels for $\mathcal{H}(@)$. Secondly, the axioms involving the $@$-operator ensure that the constant interpretations of nominals cannot be $\bot$, simplifying the proof significantly.

We first give the statement of the main result together with its proof, and subsequently prove the lemmas needed.

\begin{theorem} \label{Thm:H@A:Comp}
For any set $\Sigma$ of $\mathcal{H}(@)$-formulas, the logic $\mathbf{H}(@)\oplus\Sigma$ is sound and complete with respect to the class of all hybrid @-algebras which validate $\Sigma$. That is to say, $\vdash_{\mathbf{H}(@)\oplus\Sigma} \phi$ iff $\models_{\mathsf{H@A}(\Sigma)} \phi \approx \top$.
\end{theorem}

\begin{proof}
Suppose $\nvdash_{\mathbf{H}(@)\oplus\Sigma} \phi$. We need to find a hybrid @-algebra $\frak{A}$ and an assignment $v$ such that $\frak{A}, v \not\models \phi \approx \top$. As in the proof of Theorem \ref{Thm:HA:Comp}, we will work with an orthodox interpretation of $\mathcal{H}(@)$ to begin with, and then construct an appropriate hybrid $@$-algebra from it.

As before, we begin with the orthodox Lindenbaum-Tarski algebra of $\mathbf{H}(@)\oplus\Sigma$ over $\mathsf{PROP}$. For simplicity, denote it by $\mathbf{A}$. In the usual way, we can show that $\mathbf{A}$ validates precisely the theorems of $\mathbf{H}(@)\oplus\Sigma$. Also, $\mathbf{A}, \nu \not\models \phi \approx \top$, where $\nu$ is the natural map taking $p$ to $[p]$. In the same way as in the proof of Theorem \ref{Thm:HA:Comp}, we can also show that $[\neg\phi] > \bot$ in $\mathbf{A}$.

Next, consider the orthodox canonical extension $\mathbf{A}^{\delta}$ of $\mathbf{A}$. Since all axioms of $\mathbf{H}(@)$ are Sahlqvist---fixing the nominal coordinate, $@$ plays the role of a unary diamond---they are canonical and hence valid in $\mathbf{A}^{\delta}$. We know that $[\neg\phi] > \bot$ in $\mathbf{A}^{\delta}$, so since $\mathbf{A}^{\delta}$ is atomic, there is an atom $d \in \mathbf{A}^{\delta}$ such that $d \leq [\neg\phi]$. 
Denote the constant interpretations of the nominals occurring in $\phi$ by $s_{\nomi_1}, s_{\nomi_2}, \ldots, s_{\nomi_m}$. Note that, for each $1 \leq i \leq m$, $s_{\nomi_i} \neq \bot$ in $\mathbf{A}^{\delta}$, for otherwise, since $\mathbf{A}^{\delta}$ validates the axioms of $\mathbf{H}(@)$, $\top = @_{s_{\nomi_i}}s_{\nomi_i} = @_{s_{\nomi_i}}\bot = \bot$, a contradiction. Hence, we also have atoms $d^1_0, d^2_0, \ldots, d^m_0$ in $\mathbf{A}^{\delta}$ such that $d^1_0 \leq s_{\nomi_1}, d^{2}_{0} \leq s_{\nomi_2}, \ldots, d^{m}_{0} \leq s_{\nomi_m}$. Now, let $d^0_0 = d$, and suppose that for each $0 \leq i \leq m$, $d^i_n$ is already defined. For each $0 \leq i \leq m$, let $d^i_{n\; +\; 1} = \Diamond^{-1}d^i_n$ and set
\begin{eqnarray*}
D_i & = & \bigvee_{n\; \in\; \mathbb{N}}d^i_n.
\end{eqnarray*}
Furthermore, let
\begin{displaymath}
D = \bigvee_{0 \leq i \leq m}D_i
\end{displaymath}
and $\mathbf{A}_D  = (A_D, \wedge^D, \vee^D, \neg^D, \bot^D, \top^D, \Diamond^D, @^D, \{s^D_{\nomj}\}_{\nomj \in \mathsf{NOM}})$,
where $A_{D} = \{a \wedge D \mid a \in A\}$, $\wedge^D$ and $\vee^D$ are the restriction of $\wedge$ and $\vee$ to $A_D$,
\begin{align*}
\neg^Da &= \neg a\wedge D & \Diamond^{D}a &= \Diamond a \wedge D & \top^D &= D \\
\bot^D &= \bot & s^D_{\nomj} &= s_{\nomj} \wedge D
\end{align*}
and, for all $\nomj \in \mathsf{NOM}$,
\[
@^{D}_{s^D_\nomj} a =
\left\{ \begin{array}{ll}
D &\mbox{if } s^D_{\nomj} \leq a\\
\bot &\mbox{if } s^D_{\nomj} \leq \neg^Da
\end{array} \right.
\]

By Lemma \ref{Lemma:@:s^D_i:atoms} each $s^D_{\nomj}$ is an atom of $\mathbf{A}^{\delta}$, so it follows that $@^{D}_{s^D_\nomj} a$ is defined for each pair $(s^D_{\nomj}, a)$ where $\nomj$ is a nominal and $a \in A_D$. By Lemma \ref{Lemma:@:A_D:closed} below, $A_D$ is closed under the above operations, so $\mathbf{A}_D$ is an algebra. Furthermore, from Lemma \ref{Lemma:@:h:homomorphism} it follows that $\mathbf{A}_D$ is an orthodox interpretation of $\mathcal{H}(@)$ and that $\mathbf{A}_D \models \Sigma^{\approx}$. To show that $\mathbf{A}_D \not\models \phi \approx \top$, let $\nu_D$: $\mathsf{PROP} \to A_D$ be the assignment given by $\nu_D(p) = h(\nu(p))$, where $h$ is the homomorphism from $\mathbf{A}$ onto $\mathbf{A}_D$ defined in Lemma \ref{Lemma:@:h:homomorphism}. Using the fact that $h$ is a homomorphism, we can show by structural induction that $\widetilde{\nu}_D(\psi) = h(\widetilde{\nu}(\psi))$ for all $\mathcal{H}(@)$-formulas $\psi$ that use variables from $\mathsf{PROP}$. Now, since $d \leq D$ and $d \leq [\neg\phi]$, $\widetilde{\nu}_D(\neg\phi) = h(\widetilde{\nu}(\neg\phi)) = \widetilde{\nu}(\neg\phi) \wedge D = [\neg\phi] \wedge D \geq d > \bot$.
Hence, $\widetilde{\nu}_D(\phi) \neq D$.

Now, if we can find a suitable designated set of atoms in $\mathbf{A}_D$, we can drop the constant interpretations and we would be done. Luckily, unlike for the language $\mathcal{H}$, the @ operator makes things easier for us as it ensures that all $s^D_{\nomj}$ are atoms of $\mathbf{A}_D$ (see Lemma \ref{Lemma:@:s^D_i:atoms}). So let $\frak{A}_D = (\mathbf{A}^{-}_D, X_{A_D})$, where $\mathbf{A}^{-}_D$ is the reduct of $\mathbf{A}_D$ obtained by omitting the constant interpretations of nominals and $X_{A_D} = \{s_{\nomj}^{D} \mid \nomj \in \mathsf{NOM}\}$. Then it follows from the foregoing that $\frak{A}_D$ is a hybrid @-algebra, and since $\Sigma$ is closed under \emph{Sorted substitution}, $\frak{A}_D \models \Sigma^{\approx}$. To show that $\frak{A}_D \not\models \phi \approx \top$, consider the assignment $\nu'_D$ which extends $\nu_D$  from $\mathsf{PROP}$ to $\mathsf{PROP} \cup \mathsf{NOM}$ by simply setting $\nu'_D(\nomj) = s_{\nomj}^{D}$ for each $\nomj \in \mathsf{NOM}$. Clearly, $\widetilde{\nu}'_D(\psi) = \widetilde{\nu}_D(\psi)$ for all $\mathcal{H}(@)$-formulas $\psi$. Hence, $\widetilde{\nu}'_D(\phi) = \widetilde{\nu}_D(\phi) \neq \widetilde{\nu}_D(\top) = \widetilde{\nu}'_D(\top)$.
\end{proof}
	
We now state and prove the lemmas used in the proof of Theorem \ref{Thm:H@A:Comp}. In what follows, $\mathbf{A}$, $\mathbf{A}^{\delta}$, $\mathbf{A}_D$, $\nu$, $\nu_D$ and $\frak{A}_D$ will be as in the proof of Theorem \ref{Thm:H@A:Comp}.	
	
\begin{lemma}\label{Lemma:@:DLeqBoxD}
$D \leq \Box D$
\end{lemma}
	
\begin{proof}
By the definition of $D$, $\Box D = \Box \bigvee_{1 \leq i \leq m} D_i$, and so, using the definition of $D_i$, we get $\Box D \geq \Box \bigvee_{1 \leq i \leq m}\left(\bigvee_{n \in \mathbb{N}}d^{i}_{n}\right)$. Hence, by the monotonicity of $\Box$, $\Box D \geq \bigvee_{1 \leq i \leq m}\left(\bigvee_{n \in \mathbb{N}} \Box  d^{i}_n \right)$, which gives $\Box D \geq \bigvee_{1 \leq i \leq m}\left[\left(\bigvee_{n \in \mathbb{N} - \{0\}}\Box\Diamond^{-1}d^{i}_{n - 1}\right) \vee \Box d^{i}_0 \right]$ by the definition of $d^{i}_{n}$. But then $\Box D \geq \bigvee_{1 \leq i \leq m}\left(\bigvee_{n \in \mathbb{N} - \{0\}} \Box  \Diamond^{-1} d^{i}_{n - 1}\right)$. Since $\Diamond^{-1}$ and $\Box$ form an adjoint pair, we have that $a \leq \Box \Diamond^{-1} a$ for all $a \in \mathbf{A}^{\delta}$, so $\Box D \geq \bigvee_{1 \leq i \leq m}\left(\bigvee_{n \in \mathbb{N} - \{0\}} d^{i}_{n - 1}\right)$. Hence, by the definition of $D_i$, $\Box D \geq \bigvee_{1 \leq i \leq m} D_{i}$, which means $\Box D \geq D$, by the definition of $D$.
\end{proof}

Using this result, we can prove the following two lemmas:

\begin{lemma} \label{Lemma:@:A_D:closed}
$A_D$ is closed under the operations $\wedge^D$, $\vee^D$, $\neg^D$, $\Diamond^D$ and $@^D$.
\end{lemma}
	
\begin{proof}
The case for $@^D$ follows since $@^D_{s^D_\nomj}a$ can, by definition, only take the values $D$ and $\bot$. The other cases are proved in the same way as in Lemma \ref{Lemma:A_D:closed} using Lemma \ref{Lemma:@:DLeqBoxD}.
\end{proof}			
			
\begin{lemma} \label{Lemma:@:h:homomorphism}
The map $h$: $A \to A_D$ defined by $h(a) = a \wedge D$ is a surjective homomorphism from $\mathbf{A}$ onto $\mathbf{A}_D$.
\end{lemma}

\begin{proof}
That $h$ is surjective is obvious. In verifying that $h$ is a homomorphism, all cases except those for $\neg$, $\Diamond$, and $@_\nomj$ are straightforward. The cases for $\neg$ and $\Diamond$ are proved in exactly the same way as in Lemma \ref{Lemma:h:homomorphism}. So we need only check $@$. Since Lemma \ref{Lemma:@:s^D_i:atoms} guarantees that all $s_{\nomj}$ are atoms, the following two cases are exhaustive:

\paragraph{Case 1: $s_{\nomj} \leq a$.} In this case, by Proposition \ref{Lemma:@:BehavesCorrectly:ortho}, $h(@_{s_\nomj} a) = h(\top) = D$. From $s_{\nomj} \leq a$, we have $s_{\nomj} \wedge D \leq a \wedge D$. Hence, $h(s_{\nomj}) \leq h(a)$, which means that $@^D_{h(s_\nomj)}h(a) = D$.

\paragraph{Case 2: $s_{\nomj} \leq \neg a$.} Here $h(@_{s_\nomj} a) = h(\bot) = \bot$. From $s_{\nomj} \leq \neg a$, we have $s_{\nomj} \wedge D \leq \neg a \wedge D$. Hence, $h(s_{\nomj}) \leq h(\neg a) = \neg^Dh(a)$, giving $@^D_{h(s_\nomj)}h(a) = \bot$.
\end{proof}

\begin{lemma} \label{Lemma:@:s^D_i:atoms}
For each $\nomi \in \mathsf{NOM}$, $s_{\nomi}^D$ an atom of $\mathbf{A}^{\delta}$, and hence of $\mathbf{A}_{D}$.
\end{lemma}

\begin{proof}
First, $s_{\nomi}^D \neq \bot$, for otherwise, since $\mathbf{A}_D$ validates the axiom (\emph{Ref}), $D = @^D_{s^D_{\nomi}}s^D_{\nomi} = @^D_{s^D_{\nomi}}\bot = \bot$, which is not possible. Now, let $a, b \in \mathit{At} \mathbf{A}^{\delta}$ such that $a \leq s^D_{\nomi}$ and $b \leq s^D_{\nomi}$. We want to show that $a = b$, so suppose that they are not equal for the sake of a contradiction. From $a \leq s^D_{\nomi}$ and $b \leq s^D_{\nomi}$, we have $a, b \leq s_{\nomi}$ and $a, b \leq D$. This means there are $n_1, n_2 \in \mathbb{N}$ and $0 \leq j_1, j_2 \leq m$ such that $a \leq (\Diamond^{-1})^{n_1}d_{0}^{j_1}$ and $b \leq (\Diamond^{-1})^{n_2}d_{0}^{j_2}$. Hence, by Lemma \ref{alD^-1sbiffblDsa: alDisblDsa}, $d_0^{j_1} \leq \Diamond^{n_1}a$ and $d_{0}^{j_2} \leq \Diamond^{n_2}b$. But since $a, b \leq s_{\nomi}$, $d_0^{j_1} \leq \Diamond^{n_1}(s_\nomi \wedge a)$ and $d_{0}^{j_2} \leq \Diamond^{n_2}(s_{\nomi} \wedge b)$. We thus have $d_0^{j_1} \leq \Diamond^{n_1}(s_\nomi \wedge a) \leq \Diamond^{n_1}@_{s_{\nomi}}a \leq @_{s_{\nomi}}a$
and
$d_0^{j_2} \leq \Diamond^{n_2}(s_\nomi \wedge b) \leq \Diamond^{n_2}@_{s_{\nomi}}b \leq @_{s_{\nomi}}b$, where the second and third inequalities in each case hold by axioms (\emph{Intro}) and (\emph{Back}), respectively.
This gives $@_{s_{\nomi}}a = \top$ and $@_{s_{\nomi}}b = \top$, for otherwise, $d_0^{j_1} = d_0^{j_2} = \bot$, a contradiction. Hence, $@_{s_{\nomi}}a \wedge @_{s_{\nomi}}b = \top$, which, by axioms ($K_@$) and (\emph{Selfdual}), means that $@_{s_{\nomi}}(a \wedge b) = \top$. But $a \neq b$, so $\bot = @_{s_{\nomi}}(a \wedge b) = \top$, which is a contradiction.
\end{proof}

\subsection{Algebraic completeness of $\mathbf{H}^{+}(@)\oplus\Sigma$}

The completeness result in this section in proven with a construction similar to that used in the previous subsection. As before, we first formulate and prove the main result, and afterwards give the lemmas needed in the proof.

\begin{theorem} \label{Thm:H+@:Comp}
For any set $\Sigma$ of $\mathcal{H}(@)$-formulas, the logic $\mathbf{H}^{+}(@)\oplus\Sigma$ is sound and complete with respect to the class of all permeated hybrid @-algebras which validate $\Sigma$. That is to say, $\vdash_{\mathbf{H}^{+}(@)\oplus\Sigma} \phi$ iff $\models_{\mathsf{PH@A}(\Sigma)} \phi \approx \top$.
\end{theorem}

\begin{proof}
Suppose $\nvdash_{\mathbf{H}^{+}(@)\oplus\Sigma} \phi$. Let $\mathsf{NOM}'$ be a denumerably infinite set of nominals disjoint from $\mathsf{NOM}$. We know from \cite[Lemma 7.25]{BdRV01} that $\neg \phi$ is contained in a $\mathbf{H}^{+}(@)\oplus\Sigma$-maximal consistent set of formulas $\Gamma$ in the language extended with \textsf{NOM}' such that
\begin{enumerate}[(i)]
\item $\Gamma$ contains at least one nominal, say $\nomi_0$, and
\item for each formula of the form $@_{\nomi}\Diamond\phi \in \Gamma$, there exists a nominal $\nomj$ such that $@_{\nomi}\Diamond\nomj \in \Gamma$ and $@_{\nomj}\phi \in \Gamma$.
\end{enumerate}

Now, consider the orthodox Lindenbaum-Tarski algebra of $\mathbf{H}^{+}(@)\oplus\Sigma$ over $\mathsf{PROP}$. Denote it by $\mathbf{A}$. In the usual way, $\mathbf{A} \models \mathbf{H}^{+}(@)\oplus\Sigma^{\approx}$, while $\mathbf{A}, \nu \not\models \phi \approx \top$, where $\nu$ is the natural map taking $p$ to $[p]$. Standard propositional reasoning shows that $[\Gamma] = \{[\gamma] \mid \gamma \in \Gamma \}$ is an ultrafilter of $\mathbf{A}$, and thus, by the finite meet property, $\bigwedge [\Gamma'] > \bot$ in $\mathbf{A}$ for every finite subset $\Gamma' \subseteq \Gamma$.

Next, consider the orthodox canonical extension $\mathbf{A}^{\delta}$ of $\mathbf{A}$. Then $\bigwedge [\Gamma] > \bot$ in $\mathbf{A}^{\delta}$, for otherwise the compactness of the embedding of $\mathbf{A}$ into $\mathbf{A}^{\delta}$ would yield a finite subset $\Gamma' \subseteq \Gamma$ such that $\bigwedge [\Gamma'] \leq \bot$ in $\mathbf{A}$.

Let $d$ in $At\mathbf{A}^{\delta}$ such that $d \leq \bigwedge [\Gamma]$ and denote it by $d^0_0$. Furthermore, let $s_{\nomi_1}, s_{\nomi_2}, \ldots, s_{\nomi_m}$ be the constant interpretations of the nominals occurring in $\phi$. Since $s_{\nomi_i} \neq \bot$ in $\mathbf{A}^{\delta}$ for each $1 \leq i \leq m$, there are atoms $d^1_0, d^2_0, \ldots, d^m_0$ in $\mathbf{A}^{\delta}$ such that $d^1_0 \leq s_{\nomi_1}, d^{2}_{0} \leq s_{\nomi_2}, \ldots, d^{m}_{0} \leq s_{\nomi_m}$. Now, using $d^0_0, d^1_0, d^2_0, \ldots, d^m_0$ define $D$ and $\mathbf{A}_{D}$ as in the proof of Theorem \ref{Thm:H@A:Comp}. In the same way as in Lemma \ref{Lemma:@:A_D:closed}, we can then show that $\mathbf{A}_{D}$ is an algebra, and furthermore, by Lemma \ref{Lemma:@:h:homomorphism}, we have that $\mathbf{A}_D \models \mathbf{H}^{+}(@)\Sigma^{\approx}$. To see that $\mathbf{A}_D \not\models \phi \approx \top$, consider the assignment $\nu_D$ defined as in the proof of Theorem \ref{Thm:H@A:Comp}. We then get that $\widetilde{\nu}_D(\bigwedge \Gamma') > \bot$ for every finite subset $\Gamma'$ of $\Gamma$. But $\{\neg\phi\}$ is a finite subset of $\Gamma$, so $\widetilde{\nu}_D(\phi) \neq \top$.

In the same way as in the proof of Lemma \ref{Lemma:@:s^D_i:atoms}, we can show that for each $\nomi \in \mathsf{NOM} \cup \mathsf{NOM}'$, $s^D_{\nomi}$ is an atom of $\mathbf{A}_D$. So let $\frak{A}_{D} = (\mathbf{A}^-_{D}, X_{A_D})$, where $\mathbf{A}^-_{D}$ is the reduct of $\mathbf{A}_{D}$ obtained by omitting the constant interpretations of nominals, and $X_{A_D} = \{s_{\nomi}^D \mid \nomi \in \mathsf{NOM} \cup \mathsf{NOM}'\}$. $X_{A_D}$ is thus clearly non-empty. Furthermore, since $\mathbf{H}^{+}(@)\oplus\Sigma$ is closed under (\emph{Sorted substitution}), it follows from the foregoing that $\frak{A}_{D} \models \mathbf{H}(@)^{+}\oplus\Sigma^{\approx}$. To see that $\frak{A}_D \not \models \phi \approx \top$, consider the assignment $\nu'_D$ which extends $\nu_D$ from $\mathsf{PROP}$ to $\mathsf{PROP} \cup \mathsf{NOM} \cup \mathsf{NOM}'$, obtained by simply setting $\nu'_D(\nomi) = s_{\nomi}^D$ for each $\nomi \in \mathsf{NOM} \cup \mathsf{NOM}'$. It is clear that $\widetilde{\nu}'_D(\psi) = \widetilde{\nu}_D(\psi)$ for all $\mathcal{H}(@)$-formulas $\psi$, and hence that $\widetilde{\nu}'_D(\phi) \neq \widetilde{\nu}'_D(\top)$. Finally, by Lemma \ref{Lemma:@:A_D:permeated}, $\frak{A}_D$ is permeated.
\end{proof}

In what follows $\Gamma$, $\mathbf{A}$, $\mathbf{A}^{\delta}$, $\mathbf{A}_D$, $\nu$, $\nu_D$ and $\frak{A}_D$ will be as in the proof of the above theorem. The following lemma is needed to prove that $\frak{A}_D$ is permeated.

\begin{lemma} \label{Lemma:existence:atom}
Let $a$ be an element of $\mathbf{A}_D$. Then $d \leq @_{s^D_{\nomi}}^{D}(\Diamond^{D})^{n}a$ implies there is a nominal $\nomj \in \mathsf{NOM} \cup \mathsf{NOM}'$ such that $s^{D}_{\nomj} \leq a$.
\end{lemma}

\begin{proof}
The proof is by induction on $n$. For $n = 0$, suppose that $d \leq @_{s^D_{\nomi}}^{D}a$. This implies that $s^D_{\nomi} \leq a$, for if not, $d \leq @_{s^D_{\nomi}}^{D}a = \bot$, which is a contradiction. Hence, $\nomi$ can be taken as the desired nominal. Now, suppose that for every $a \in A_D$, the claim holds for $n = k$. For $n = k + 1$, assume $d \leq @^{D}_{s^D_{\nomi}}(\Diamond^{D})^{k + 1}a = @^{D}_{s^D_{\nomi}}(\Diamond^{D})^{k}\Diamond^D a$.
By the inductive hypothesis, we have a nominal $\nomk \in \mathsf{NOM} \cup \mathsf{NOM}'$ such that $s^{D}_{\nomk} \leq \Diamond^{D}a$. This measn that that $D = @^{D}_{s^D_{\nomk}}s^{D}_{\nomk} \leq @^{D}_{s^D_{\nomk}}\Diamond^{D}a$, and so $d \leq @^D_{s^D_{\nomk}}\Diamond^Da$. Now, since both $\nu$ and $h$ are surjective, there is a $\psi$ such that $d \leq @^D_{s^D_{\nomk}}\Diamond^D\widetilde{\nu}_D(\psi) = \widetilde{\nu}_D(@_{\nomk}\Diamond\psi)$. In the same way as in Lemma \ref{Lemma:gammainGamma:dleq}, we can show that for every $\mathcal{H}(@)$-formula $\gamma$, $\gamma \in \Gamma$ iff $d \leq \widetilde{\nu}_D(\gamma)$. But then $@_{\nomk}\Diamond\psi \in \Gamma$, which means there is a nominal $\nomj \in \mathsf{NOM} \cup \mathsf{NOM}'$ such that $@_{\nomk}\Diamond\nomj \in \Gamma$ and $@_{\nomj}\psi \in \Gamma$. Hence, $d \leq \widetilde{\nu}_{D}(@_{\nomj}\psi) = @^{D}_{s^D_{\nomj}}\widetilde{\nu}_{D}(\psi) = @^D_{s^D_{\nomj}}a$, and so $D = @^D_{s^D_{\nomi_0}}d \leq @_{s^D_{\nomi_0}}^{D}@^{D}_{s^D_{\nomj}}a \leq @^{D}_{s^D_{\nomj}}a$, where the last inequality follows from axiom (\emph{Agree}). Therefore, $@^{D}_{s^D_{\nomj}}a = D$, which implies that $s_{\nomj}^{D} \leq a$.
\end{proof}

\begin{lemma} \label{Lemma:@:A_D:permeated}
$\frak{A}_{D}$ is permeated.
\end{lemma}
\begin{proof}
To prove the first condition, let $b \in \mathbf{A}_{D}$ such that $b > \bot$. Then
\[
b \wedge \bigvee_{0 \leq i \leq m}D_i > \bot,
\]
which means there is a $0 \leq j \leq m$ such that $b \wedge D_j > \bot$. Hence,
\[
b \wedge \bigvee_{n \in \mathbb{N}}(\Diamond^{-1})^nd^j_0 > \bot,
\]
and so $b \wedge (\Diamond^{-1})^{n_j}d^j_{0} > \bot$ for some $n_j \in \mathbb{N}$. We thus know from Lemma \ref{inverse>bot:dleqdiamomd} that $d^j_0 \leq \Diamond^{n_j}b$, and therefore $d^j_0 \leq (\Diamond^D)^{n_j}b$ (the proof of this is similar to that of Lemma \ref{dleqdiamond:dleqdiamondd}). Now, by construction, $d^j_0 \leq s_{\nomi_j}$, and hence $d^j_0 \leq s^D_{\nomi_j}$. But both $d^j_0$ and $s^D_{\nomi_j}$ are atoms, so $d^j_0 = s^D_{\nomi_j}$. Hence, $D = @^D_{s^D_{\nomi_j}}s^D_{\nomi_j} \leq @^D_{s^D_{\nomi_j}}(\Diamond^D)^{n_j}b$, which means that $d \leq @^D_{s^D_{\nomi_j}}\Diamond^{n_j}b$. It thus follows from Lemma \ref{Lemma:existence:atom} that there is a nominal $\nomj$ such that $s^D_{\nomj} \in X_{A_D}$ and $s^D_{\nomj} \leq b$.

For the second condition, let $s^D_{\nomi} \in X_{A_D}$ and $b \in A_D$ such that $s^D_{\nomi} \leq \Diamond^Db$. Then we have $D = @_{s_{\nomi}^{D}}^{D}\Diamond^{D}b$. By the surjectivity of $\nu$ and $h$ there is a formula $\psi$ such that $b = \widetilde{\nu}_{D}(\psi)$, and hence $d \leq D = @_{s^D_{\nomi}}^{D}\Diamond^{D}b = @_{s^D_{\nomi}}^{D}\Diamond^{D}\widetilde{\nu}_{D}(\psi) = \widetilde{\nu}_{D}(@_{\nomi}\Diamond\psi)$. This means that $@_{\nomi}\Diamond\psi \in \Gamma$, and therefore there is a nominal $\nomj \in \mathsf{NOM} \cup \mathsf{NOM}'$ such that $@_{\nomi}\Diamond\nomj \in \Gamma$ and $@_{\nomj}\psi \in \Gamma$. Hence, $d \leq \widetilde{\nu}_{D}(@_\nomi\Diamond\nomj) = @_{s^D_{\nomi}}^{D}\Diamond^{D}s_{\nomj}^{D}$ and $d \leq \widetilde{\nu}_{D}(@_{\nomj}\psi) = @_{s^D_{\nomj}}^{D}b$. It follows from the definition of $@^D$ that $s_{\nomj}^{D} \leq b$ and $s_{\nomi}^{D} \leq \Diamond^{D}s_{\nomj}^{D}$, as desired.
\end{proof}

\section{Conclusion}\label{sec:Conc}

We introduced the notions of hybrid algebras and hybrid $@$-algebras together with their permeated subclasses. We considered some basic truth preserving operations for these structures and explored their duality with two-sorted general frames. We proved general completeness results for axiomatic extensions of the basic hybrid logics with respect to classes of hybrid algebras.

The availability of this semantics, so close to the familiar algebraic semantics for modal logic, has proven extremely useful to the authors in their study of hybrid logic. It has enabled us to adapt existing algebraic techniques for modal logic to obtain new result for hybrid logics, including Sahlqvist-type theorems \cite{ConRob2015} and finite model properties \cite{ConRobBull, Robinson:PhD:Thesis}.

In this paper we have considered only two hybrid languages, and have not mentioned languages containing e.g. the universal modality of the `down-arrow' binder.  In is in fact easy to extend the results obtained to languages like $\mathcal{H}(\mathsf{E})$ with the universal operator, using similar constructions and ideas to the one we have employed. This is worked out in full detail in \cite{Robinson:PhD:Thesis}. As already mentioned, the `down-arrow' binder seems to require a very different treatment (see \cite{Litak06:relmics}).

\bibliographystyle{alpha}
\bibliography{HyComp}

\end{document}